\documentclass{article}
\usepackage{amsthm}

\usepackage{arxiv}

\usepackage{iftex}
\ifPDFTeX
  \usepackage[utf8]{inputenc} 
  \usepackage[T1]{fontenc}    
\else
  \usepackage{fontspec}       
  \defaultfontfeatures{Ligatures=TeX}
\fi
\usepackage{hyperref}       
\usepackage{url}            
\usepackage{booktabs}       
\usepackage{amsmath}        
\usepackage{amsfonts}       
\usepackage{amssymb}        
\usepackage{calc}           
\usepackage{subcaption}     
\usepackage{csquotes}       
\usepackage{nicefrac}       
\usepackage{cleveref}       
\newtheorem{question}{Question}
\newtheorem{theorem}{Theorem}
\newtheorem{definition}{Definition}
\newtheorem{corollary}{Corollary}
\newtheorem{proposition}{Proposition}
\newtheorem{conjecture}{Conjecture}
\newtheorem{lemma}{Lemma}

\newtheorem{observation}{Observation}
\newtheorem{remark}{Remark}

\crefname{theorem}{Theorem}{Theorems}
\crefname{lemma}{Lemma}{Lemmas}
\crefname{definition}{Definition}{Definitions}
\crefname{proposition}{Proposition}{Propositions}
\crefname{question}{Question}{Questions}
\crefname{corollary}{Corollary}{Corollaries}
\crefname{conjecture}{Conjecture}{Conjectures}
\crefname{remark}{Remark}{Remarks}
\crefname{observation}{Observation}{Observations}
\crefname{equation}{Eq.}{Eqs.}
\usepackage{microtype}      
\usepackage{graphicx}
\usepackage{enumitem}

\usepackage[
  backend=biber,
  style=ieee,
  date=year,
  natbib=true,
  labeldate=year,
  alldates=year,
  urldate=short,
  maxnames=99,
  minnames=99,
  maxbibnames=99, 
  minbibnames=99,  
  sortcites=true,
  sorting=nyt,
  url=true,
]{biblatex}

\AtEveryBibitem{\clearfield{keywords}
  \clearfield{note}
  \clearfield{annotation}
}
\AtEveryBibitem{\clearfield{month}}
\AtEveryBibitem{\clearfield{day}}

\AtEveryBibitem{
  \ifentrytype{article}{
    \clearfield{url}
    \clearfield{urlyear}
  }{}
}
\AtEveryBibitem{
  \ifentrytype{inproceedings}{
    \clearfield{url}
    \clearfield{urlyear}
  }{}
}
\AtEveryBibitem{
  \ifentrytype{misc}{
    \clearfield{url}
    \clearfield{urlyear}
  }{}
}

\usepackage{doi}

\usepackage{tikz}
\usetikzlibrary{calc}

\usepackage{listings}
\usepackage{textcomp}    
\usepackage{upquote}     

\usepackage{algorithm}
\usepackage{algpseudocode}

\usepackage[dvipsnames]{xcolor}

\definecolor{codegreen}{rgb}{0,0.6,0}
\definecolor{codegray}{rgb}{0.5,0.5,0.5}
\definecolor{codepurple}{rgb}{0.58,0,0.82}
\definecolor{backcolour}{rgb}{0.95,0.95,0.92}

\lstdefinestyle{julia}{
  backgroundcolor=\color{backcolour},
  commentstyle=\color{codegreen},
  keywordstyle=\color{magenta},
  numberstyle=\tiny\color{codegray},
  stringstyle=\color{codepurple},
  basicstyle=\ttfamily\footnotesize,
  breakatwhitespace=false,
  breaklines=true,
  captionpos=b,
  keepspaces=true,
  numbers=left,
  numbersep=5pt,
  showspaces=false,
  showstringspaces=false,
  showtabs=false,
  tabsize=2,
  language=Julia
}

\usepackage{etoolbox}
\usepackage{epigraph}

\usepackage{rotating}

\DeclareMathOperator{\Aut}{Aut}
\DeclareMathOperator{\PAut}{PAut}
\DeclareMathOperator{\ran}{ran}
\DeclareMathOperator{\dom}{dom}
\DeclareMathOperator{\fix}{fix}
\DeclareMathOperator{\supp}{supp}
\DeclareMathOperator{\dual}{dual}
\DeclareMathOperator{\rank}{rank}
\DeclareMathOperator{\degen}{degen}
\DeclareMathOperator{\ex}{ex}
\DeclareMathOperator{\dist}{dist}
\DeclareMathOperator{\id}{id}



\usepackage{pdflscape}
\usepackage{algorithm}

\title{Extremal Asymmetric Depth of Planar Graphs and Hidden Near-Mirror Symmetries of IPR Fullerenes}

\author{%
  \href{https://orcid.org/0000-0001-8237-1275}{\includegraphics[scale=0.06]{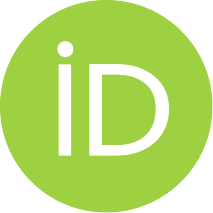}\hspace{1mm}Ján~Pastorek}\thanks{Corresponding author.} \\
  Department of Applied Informatics, Faculty of Mathematics, Physics and Informatics\\
  Comenius University Bratislava\\
  Mlynská dolina F1, 842~48 Bratislava, Slovakia\\
  \texttt{jan.pastorek@fmph.uniba.sk} \\
}

\renewcommand{\shorttitle}{Asymmetric Depth and Hidden Near-Mirror Symmetries of IPR Fullerenes}

\hypersetup{
  pdftitle={Extremal Asymmetric Depth of Planar Graphs and Hidden Near-Mirror Symmetries of IPR Fullerenes},
  pdfsubject={graph theory, planar graphs, fullerenes},
  pdfauthor={Jan~Pastorek},
  pdfkeywords={asymmetric graphs, partial automorphisms, asymmetric depth, planar graphs, fullerenes, IPR fullerenes},
}

\begin{document}
\maketitle

\begin{abstract}
   Although almost all graphs are asymmetric---having no nontrivial global
  automorphisms---they may still possess local symmetries in the form of
  isomorphisms between induced subgraphs, i.e., partial automorphisms. We study
  such local symmetries via \emph{asymmetric depth}, defined in terms of the
  maximum rank of a nontrivial partial automorphism. We prove a tight upper
  bound on asymmetric depth in the class of planar graphs and identify the
  extremal graphs: duals of IPR fullerenes attain the maximum already on $47$
  vertices. Our main structural result concerns the IPR fullerenes that are neither maximally asymmetric nor symmetric. In such a cage no purely local action realises a low asymmetric
  depth, and we show that the map which does realise it cannot be confined to a
  small part of the cage either: neither to a single face, nor behind an
  interface of at most $5-k$ edges, $k \le 3$ being the deficiency. A cage of
  asymmetric depth $2$ or $3$ is therefore not asymmetric in one place; it
  carries a broken symmetry invisible to its automorphism group. Such cages are
  rare---under $2\%$ of the asymmetric IPR fullerenes at $n = 118$. In all $727$ of them the largest partial automorphism is a
  near-mirror reflection, which we state as an explicit conjecture. We also extend the asymmetric depth bound to
  graphs of higher genus.
\end{abstract}


\medskip
\noindent\textbf{Math. Subj. Class. (2020):} 05C10, 05C25, 92E10, 20M18.

\section{Motivation and Background}

The structural information one extracts from the automorphism group of a graph diminishes as the number of vertex orbits grows. In the extreme case, knowing that the automorphism group of a graph is trivial yields almost no information about the structure of the graph. Yet almost all graphs are known to be \textit{asymmetric}~\cite{erdosAsymmetricGraphs1963}; that is, they have no nontrivial ``global'' automorphisms. At the same time, all graphs contain at least some local symmetries. \citet{kimAsymmetryRandomRegular2002} showed that the same is true for the class of regular graphs, which contains some of the most symmetric graphs, namely, the vertex-transitive graphs. Taking the opposite point of view, deleting a single vertex from a (vertex-transitive) graphical regular representation leads to an asymmetric graph. However, it is important to note that the distorted graph obtained this way still has many isomorphic induced subgraphs with nontrivial automorphisms, suggesting that the line between symmetric and asymmetric objects is surprisingly thin.

Perhaps for this reason, researchers have studied more specific combinatorial and extremal families of asymmetric graphs and their properties \cite{brewerAsymmetricIndexGraph2020,schweitzerMinimalAsymmetricGraphs2017,spencerMaximalAsymmetryGraphs1976,quintasExtremaConcerningAsymmetric1967,feigeRobustlyAsymmetricGraphs2014,lepovicStronglyAsymmetricGraphs1995,nesetrilStructureAsymmetricGraphs}. In particular, in \cite{jajcay_inverse_2021,cingelPartialAutomorphismsLevel2024}, the authors suggested an approach to studying graphs on a more granular level using partial automorphisms of graphs, that is, isomorphisms between induced subgraphs. This approach can be viewed as a generalization of group theory: the set of all partial automorphisms of a finite graph, equipped with the usual composition of partial maps, forms the partial automorphism inverse monoid of the graph. In \cite{cingelPartialAutomorphismsLevel2024}, a measure quantifying ``how far is any graph from having a symmetry?'' was defined in terms of the rank of the largest nontrivial partial automorphism of the graph. This measure is motivated by the algebraic structure of partial automorphism inverse monoids of graphs.

Fullerene graphs are mathematical models of fullerene molecules, that is, carbon cages distinct from graphite and diamond. From a graph-theoretical perspective, the study of fullerenes has been driven by the search for invariants that correlate with the stability of a chemical compound \cite{andovaMathematicalAspectsFullerenes2016}. The fullerene graphs in which no two pentagons are adjacent, that is, each pentagon is surrounded by five hexagons, satisfy the \textit{isolated pentagon rule} (IPR). IPR fullerenes have been shown to be stable fullerene compounds \cite{krotoStabilityFullerenesCn1987,austinStructuralMotifsStability1995,schmalzElementalCarbonCages1988}. Fullerenes that violate this rule have been deemed too reactive to be synthesized \cite{tanStabilizationFusedpentagonFullerene2009}. \citet{raghavachariGroundStateC841992} suggested that steric strain is reduced when the pentagons are distributed as evenly as possible, and on this basis formulated the uniform curvature rule as an extension of the IPR principle. More recently, \citet{rodriguez-forteaMaximumPentagonSeparation2010} introduced the maximum pentagon separation rule, arguing that the most favorable carbon cages are those in which the $12$ pentagons are separated as much as possible. In \cite{goedgebeurFullerenesDistantPentagons2015}, the authors investigated the smallest fullerenes for which the minimum distance between two pentagons is fixed. These and other results show that local neighbourhoods and patches are of paramount importance in the study of fullerenes.

On the one hand, fullerenes are known to have only $28$ possible symmetry groups \cite{fowlerPossibleSymmetriesFullerene1993}. On the other hand, as shown in Table~\ref{table:asym_ipr}, most IPR fullerenes are asymmetric. The counts in Table~\ref{table:asym_ipr} report, for each tested size $n$, the number of IPR fullerenes generated by \texttt{buckygen} and the number among them that are asymmetric. Over the computed range the proportion of asymmetric IPR fullerenes trends upward, though not monotonically---it dips at $n = 88, 96, 104, 108, 112$ and $116$---and exceeds $94\%$ at $n = 118$; this is consistent with a tendency toward $1$, for which we are not aware of a proof.

While the global automorphism group of a large, complex fullerene often reduces to the trivial group, its partial automorphism inverse monoid provides a tool for studying local symmetries via isomorphisms between induced subgraphs, offering much more granularity.

In this paper, we take steps toward bringing the partial-automorphisms point of view to planar and fullerene graphs by quantifying the absence of large induced-subgraph isomorphisms (partial automorphisms). We call this measure asymmetric depth and define it in the next section.

A word on why this point of view might interest a chemist. The vertex orbits of $\Aut(F)$ are exactly the sets of carbon sites of the cage that are equivalent by symmetry; by Mani's theorem~\cite{mani1971automorphismen} the automorphism group of a $3$-connected planar graph is realised by the isometry group of a convex polyhedron with that graph as its skeleton, so for a fullerene these orbits are also the symmetry-equivalent sites of an idealised molecular geometry. Counting them is how $^{13}$C NMR spectra are read: $\mathrm{C}_{60}$ has a single line and $\mathrm{C}_{70}$, whose $70$ carbons fall into five orbits, has five. For an asymmetric cage there are $n$ orbits and this reading yields nothing. A partial automorphism of rank $n-k$ instead matches all but $k$ of the sites to one another, so asymmetric depth measures how far the cage is from having a symmetry, counted in atoms rather than in the continuous geometric distance of the symmetry measures of \citet{zabrodskyContinuousSymmetryMeasures1992,pinskyContinuousSymmetry1998}; unlike those, it is insensitive to conformation and is computed exactly. Read this way, the main theorem of \Cref{sec:hidden} says that when a cage is only two or three atoms away from a symmetry, there is never a small defective patch to point at: the near-symmetry property belongs to the whole cage. We offer asymmetric depth in this spirit, as a new perspective on near-symmetry in carbon cages rather than as a descriptor with an established chemical correlate.

\section{Preliminaries}

All graphs considered in this paper are finite, undirected and simple.

Let $F = (V, E)$ be a connected graph. The \emph{closed neighbourhood} of a vertex $v \in V$ is the set $N[v] = \{v\} \cup \{u \in V : uv \in E\}$, and the \emph{open neighbourhood} of $v$ is the set $N(v) = N[v] \setminus \{v\}$. For $S \subseteq V(F)$, we write $F[S]$ for the subgraph of $F$ \emph{induced} by~$S$, that is, the graph with vertex set $S$ whose edges are exactly the edges of $F$ with both endpoints in~$S$.

A subset of vertices $S \subseteq V(F)$ is called a \emph{vertex cut} of $F$ if the vertex-deleted subgraph $F - S$ is disconnected, or if $F - S$ is a trivial graph consisting of a single vertex. A graph is $k$-connected if the minimum number of vertices whose removal disconnects the graph is at least $k$.

A subset of edges $X \subseteq E$ is called an \emph{edge cut} if the edge-deleted subgraph $F - X$ is disconnected. An edge cut $X \subseteq E$ is called a \emph{cyclic edge cut} if at least two of the connected components of $F - X$ contain a cycle. For a vertex set $S \subseteq V(F)$, we write $\partial S$ for the \emph{edge boundary} of~$S$, i.e., the set of edges of $F$ with exactly one endpoint in~$S$.

A cycle is called \textit{separating} if the removal of its vertices increases the number of connected components of the graph.

\subsection{Planar and fullerene graphs}

A graph $F$ is \textit{planar} if it can be embedded in the plane so that edges intersect only at their endpoints.
When we refer to faces or to the dual graph, we implicitly fix such an embedding and treat $F$ as a \emph{plane} graph (i.e., a planar graph together with a fixed embedding).

Recall that the \textit{dual} graph of a plane graph $F$, $\dual(F)$, is a graph with vertices corresponding to the faces of $F$; two vertices in $\dual(F)$ are adjacent if the corresponding faces in $F$ share an edge. If $F$ is 3-connected planar, then $\dual(F)$ is well-defined up to isomorphism. We will also write $F^* := \dual(F)$.

Let $F$ be a plane graph. An edge $e$ of $F$ is called \textit{weak} if $e$ is incident with two triangular faces, and it is called \textit{semiweak} if $e$ is incident with only one triangular face. The \textit{weight} of an edge is the degree sum of its end vertices.

\textit{Fullerene} graphs are cubic, 3-connected, planar graphs with only pentagonal and hexagonal faces. Fullerene graphs are cyclically $5$-edge-connected~\cite{cyclic_edge_cuts}; equivalently, no nontrivial cyclic edge cut has fewer than $5$ edges. A fullerene is called \textit{IPR} if it has no adjacent pentagons. Euler's formula implies that a fullerene on $n$ vertices has exactly $12$ pentagonal and $n/2-10$ hexagonal faces. Its dual is obtained by exchanging vertices and faces: each face of the fullerene becomes a vertex of the dual, and two dual vertices are adjacent exactly when the corresponding faces share an edge. Consequently, the dual of a fullerene on $n$ vertices is a triangulation (that is, a planar graph in which every face is a triangle) with $12$ vertices of degree $5$ and $n/2-10$ vertices of degree $6$. The face distance between two pentagons is the graph distance between the corresponding degree-$5$ vertices in the dual graph.

The cyclic edge-connectivity of a fullerene graph cannot exceed 5, since it contains 12 pentagons, thus, there are at least 12 cyclic 5-edge-cuts formed by the edges pointing outwards of each pentagonal face. There are also cyclic 6-edge-cuts formed by the edges pointing outwards of each hexagonal face. These cyclic 5- and 6-edge-cuts will be called trivial.

\subsection{Partial automorphisms and asymmetric graphs}

We call a graph \emph{symmetric} if it possesses at least one nontrivial automorphism, i.e., if $\Aut(F)\neq 1$.
We study local structure and local symmetries using the following tools.

A \emph{partial permutation} of $X$ is a bijective mapping \(\varphi\colon A \to B\) from a subset \(A \subseteq X\) onto a subset \(B \subseteq X\). We call \(A\) the \emph{domain} of \(\varphi\) (written \(\dom(\varphi)\)) and \(B\) the \emph{range} of \(\varphi\) (written \(\ran(\varphi)\)). The rank of a partial permutation is the size of its domain or equivalently the size of its range. We denote any restriction of a partial permutation $\varphi$ with domain $X$ to a subset $S \subseteq X$ by $\varphi_{|S}$. The set of all partial permutations on a set $X$, closed under composition and inverse of partial maps, forms a \emph{symmetric inverse monoid}.

A \textit{partial automorphism} $\varphi$ of a graph $F$ is an isomorphism between induced subgraphs of $F$. Every partial automorphism of a graph is a partial permutation on the set of its vertices. Some of the basic results on structure of partial automorphisms can be found in \cite{jajcayovaComputationalAspects2022,Gal2023,cingelPartialSymmetriesSymmetry2023,cingelPartialAutomorphismsLevel2024}. Let us denote by $\fix(\varphi) := \{v \in \dom(\varphi) : \varphi(v) = v\}$ the set of vertices fixed by a partial automorphism $\varphi$, and by
\[
  \supp(\varphi) \;:=\; \{v \in \dom(\varphi) : \varphi(v) \ne v\} \;\cup\; \{v \in \ran(\varphi) : \varphi^{-1}(v) \ne v\}
\]
the \emph{support} of $\varphi$, i.e., the set of vertices moved by $\varphi$ or by its inverse. With this symmetric convention $\supp(\varphi) = \supp(\varphi^{-1})$, and the image of every non-fixed domain vertex again belongs to $\supp(\varphi)$; both properties are used repeatedly in \Cref{sec:hidden}. Note that every $x \in \supp(\varphi)$ either satisfies $x \in \dom(\varphi)$ and $\varphi(x) \neq x$, or lies outside $\dom(\varphi)$ (and is then the image of a non-fixed domain vertex).

The set of all partial automorphisms of a graph $F$ together with the operations of partial composition and partial
inverse of partial maps, forms an inverse monoid, called \textit{partial automorphism inverse monoid} and denoted as \(\PAut(F)\). Any partial automorphism inverse monoid of a graph is an inverse submonoid of symmetric inverse monoid on the set of vertices. Partial automorphism inverse monoids were fully characterized for graphs by \citet{jajcay_inverse_2021}.

The partial automorphism monoid of a graph $F$, $\PAut(F)$, is a complex algebraic structure that contains the automorphism group of $F$, $\Aut(F)$, the automorphism group of any of its induced subgraphs, and all isomorphisms between induced subgraphs. Computing the entire partial automorphism monoid is a computationally intensive task~\cite{jajcayovaComputationalAspects2022}.

However, we want to study asymmetric graphs. One can observe that if there was a guarantee that the graph on input is asymmetric and has no non-trivial partial automorphisms of rank higher than $k$, then we know quite a lot about the structure of the monoid above that rank. All induced subgraphs of order $n$ down to $k$ would be pairwise non-isomorphic and asymmetric.

We call a partial automorphism $\varphi$ \emph{trivial} if it is a restriction of the identity map, i.e., $\varphi(x) = x$ for every $x \in \dom(\varphi)$ (equivalently $\supp(\varphi) = \emptyset$), and \emph{nontrivial} otherwise; thus $\varphi$ is nontrivial precisely when it moves at least one vertex. In particular, a nontrivial partial automorphism of rank $n$ is exactly a nontrivial automorphism of $F$. Note that every automorphism of $F$ is a partial automorphism, namely one of rank $n$ whose domain and range are all of $V(F)$. For $n \ge 2$ every graph admits some nontrivial partial automorphism (for instance the local transposition of any two vertices, see \Cref{lem:swap}), so the quantity in the next definition is well defined.

Thus, in \cite{cingelPartialSymmetriesSymmetry2023} and \cite{cingelPartialAutomorphismsLevel2024} the following concepts were introduced for the study of partial automorphism inverse monoids of asymmetric graphs.

\begin{definition}[\em Asymmetric depth]
Let $F$ be a graph of order $n \geq 2$, and let $k$ be the largest positive integer for which $F$ admits a nontrivial partial automorphism $\varphi$ of rank $k$. The \textit{asymmetric depth} of $F$ is $d(F) := n - k$.
\end{definition}

Note that $d(F) = 0$ if and only if $F$ is symmetric: a nontrivial automorphism is a nontrivial partial automorphism of rank $n$, giving $k = n$, and conversely $k = n$ means $F$ has a nontrivial automorphism. For instance, if a graph $F$ has $d(F) = 3$, then all induced subgraphs of orders $n$, $n-1$, and $n-2$ are asymmetric and pairwise non-isomorphic. Any asymmetric graph $F$ on $6$ vertices has $d(F) = 1$, as there are no smaller asymmetric graphs. By the result of \citet{schweitzerMinimalAsymmetricGraphs2017}, there are precisely $18$ minimally asymmetric graphs, i.e., asymmetric graphs whose every induced subgraph has a nontrivial automorphism. All of these graphs have asymmetric depth equal to $1$ and almost all of them are planar. One of the smallest asymmetric $3$-regular planar graphs, Frucht's graph $F$, on $12$ vertices has $d(F) = 2$. In other words, Frucht's graph has no two isomorphic induced subgraphs of order $11$, and each of its induced subgraphs of order $11$ is asymmetric.

\paragraph{Relation to prior work.}
In \cite{cingelPartialSymmetriesSymmetry2023}, the authors posed an analogue of the following question we address in this paper for the class of planar graphs: ``What is the maximal asymmetric depth of a graph $F$ of order $n$?''

A tight upper bound was proved in \cite{cingelPartialAutomorphismsLevel2024}: for every graph $F$ of order $n$, $d(F) \leq \left\lfloor\frac{n-1}{2}\right\rfloor$. The proof relies on the existence of the following partial automorphisms:

\begin{lemma}[Local transposition, {\cite{cingelPartialAutomorphismsLevel2024}}]\label{lem:swap}
Let $F$ be a graph of order $n\ge 2$ and let $u,v\in V(F)$ be distinct.
Write
\[
  \Delta_{uv}:= (N(u) \setminus \{  v\} ) \; \Delta \; (N(v) \setminus \{  u\})
\]
for the symmetric difference of the open neighbourhoods of $u$ and $v$, each with the other vertex excluded.
Then $F$ admits a (nontrivial) partial automorphism of rank $n-|\Delta_{uv}|$, namely the map that swaps $u$ and $v$ and fixes every vertex of $V(F)\setminus(\Delta_{uv}\cup\{u,v\})$.
\end{lemma}

\begin{corollary}[{\cite{cingelPartialAutomorphismsLevel2024}}]\label{cor:basic-ineq}
For every graph $F$,
\begin{equation}\label{basic-ineq}
  d(F) \leq \min_{u\neq v} \{|\Delta_{uv}|\}.
\end{equation}
where the minimum is taken over all unordered pairs of distinct vertices $u,v$ of $F$.
\end{corollary}

For any cubic graph (and hence for any fullerene graph), every two adjacent vertices $u$ and $v$ satisfy $|\Delta_{uv}|\le 4$. Consequently, by \Cref{cor:basic-ineq}, every fullerene graph $F$ satisfies $d(F)\le 4$.

The notion of asymmetric depth used here was introduced in \cite{cingelPartialSymmetriesSymmetry2023,cingelPartialAutomorphismsLevel2024} as a way to quantify asymmetricity of graphs.

Asymmetric depth sits alongside several established measures of how symmetry can be created or destroyed by small modifications of a graph. The distinguishing number~\cite{albertsonCollinsSymmetryBreaking1996} and the determining number~\cite{boutinDeterminingSets2006} measure how much colouring, respectively how many fixed vertices, are needed to \emph{destroy} all symmetry of a symmetric graph; asymmetric depth measures, dually, how much of an asymmetric graph must be disregarded before a symmetry \emph{appears}. Closest to the present work is \citet{aksionovDeeplyAsymmetricPlanar2005}, whose structural theorem we use in \Cref{sec:bridge}: those authors studied precisely this second question for planar graphs, but with \emph{edge} modifications in place of vertex modifications, proving that deleting at most five edges reduces every planar graph of order at least two to a graph with a nontrivial automorphism, and that five cannot be lowered to four. \Cref{t:planar} below is the vertex analogue, and it is striking that the extremal constant is again $5$; whether the two quantities are related beyond their shared reliance on the same discharging argument we do not know. 

In \cite{cingelPartialAutomorphismsLevel2024}, the authors established a tight general upper bound on the maximal depth of graphs. Moreover, they proved the following bound that for any simple planar graph $F$ of order $n \geq 2$,

\begin{equation}\label{c:planar}
  d(F) \;\le\;\frac{12n-24}{\,n\,}\;-\;\frac{8(3n-6)^2}{n^2(n-1)} \;=\; 12 - O(1/n)
\end{equation}
without providing examples attaining this bound. \Cref{t:planar} below replaces $12 - O(1/n)$ by the constant $5$, and \Cref{thm:smallest-depth5} exhibits graphs attaining it.

\paragraph{Contribution.}

Our main theorem is a rigidity statement for the partial symmetries of IPR fullerenes. Every pair of vertices of an IPR fullerene satisfies $|\Delta_{uv}| \ge 4$, so a local transposition never realises asymmetric depth below~$4$ (\Cref{cor:basic-ineq}); an IPR fullerene of asymmetric depth $2$ or $3$ must therefore carry a large nontrivial partial automorphism that is not a transposition. We prove that no such map can be \emph{localised}: after deleting at most three vertices of an IPR fullerene, a nontrivial partial automorphism can never confine its nontrivial action to a single face, nor to any region joined to the pointwise-fixed remainder of the cage across a small edge interface (\Cref{thm:unified-localisation}). In particular, these rare cages carry a hidden, almost-global symmetry that the automorphism group cannot see (\Cref{cor:hidden-symmetry-primal}). Far from being sporadic, this behaviour is universal in the low-depth regime: inspecting our computational data, we find that \emph{every} IPR fullerene of asymmetric depth $2$ or $3$ up to $n = 118$ vertices exhibits such a hidden symmetry, and in each case it is a near-mirror one. \Cref{fig:asym_depth_2} shows the smallest example---the depth-$2$ IPR fullerene on $92$ vertices---and portrays the almost-global behaviour: the realising partial automorphism is a near-mirror reflection that fixes an axis of vertices, omits only two from its domain and range, and moves support across many faces of the cage. The theorem assumes neither the asymmetry of the cage nor the extremality of the map, and we show it is sharp in a precise sense: it fails for cubic $3$-connected graphs of girth~$5$ once planarity is dropped (\Cref{prop:girth5-witness}). We stress at the outset that the low-depth regime governed by this theorem is exceptional: over our census (\Cref{table:asym_ipr}) the IPR fullerenes of asymmetric depth $2$ or $3$ form a small and, across the computed range, decreasing fraction of all asymmetric IPR fullerenes. Most asymmetric fullerenes have asymmetric depth~$4$ realised by a single transposition. The rigidity below is therefore a statement about the rare cages that escape this generic behaviour.

\begin{itemize}[leftmargin=2em]
\item \textbf{Main theorem.} In an IPR fullerene, no nontrivial partial automorphism of deficiency at most~$3$ acts nontrivially only inside a single face or behind an interface of at most a few edges (\Cref{thm:unified-localisation}); in particular the support of every partial automorphism realising asymmetric depth $2$ or $3$ is neither contained in a single face nor separated from the pointwise-fixed part of the cage by at most $5-k$ edges (\Cref{cor:hidden-symmetry-primal}).
\item We establish a tight upper bound of $5$ on the asymmetric depth of planar graphs, and give explicit duals of IPR fullerenes attaining it.
\item We determine the possible asymmetric depths of fullerene graphs, $d(F) \in \{0,2,3,4\}$, and provide exhaustive computational data for IPR fullerenes and their duals.
\item We extend the upper bound on asymmetric depth to graphs of higher genus.
\end{itemize}

\section{Planar graphs: upper bound}\label{sec:bridge}

Here we improve the unpublished bound in \eqref{c:planar}, using a theorem proved via a discharging argument in \citet{aksionovDeeplyAsymmetricPlanar2005}. We show that any planar graph has asymmetric depth at most $5$. The improvement is short: the discharging is entirely Aksionov's, and once his theorem is in hand the bound follows from a two-line neighbourhood count (\Cref{cor:aksionov}); the new content of this and the next section is the identification of the extremal graphs, namely the duals of IPR fullerenes.

\begin{theorem}[\cite{aksionovDeeplyAsymmetricPlanar2005}]
Every connected planar graph with at least two vertices contains two vertices $u,v$ that satisfy at least one of the following conditions:
(a) $\deg(u)+\deg(v)\le 5$;
(b) $\dist(u,v)\in\{1,2\}$ and $\deg(u)+\deg(v)\le 7$;
(c) a weak edge of weight at most $11$; or
(d) a semiweak edge of weight at most $9$.
\end{theorem}

\begin{figure}[!ht]
    \centering
\begin{tikzpicture}[
    vertex/.style={circle, draw=black, thick, inner sep=1pt, minimum size=4.5mm, font=\sffamily\scriptsize},
    uv/.style={vertex, fill=red!40},
    shared/.style={vertex, fill=gray!30},
    delta/.style={vertex, fill=blue!30},
    edge/.style={thick, draw=black!80},
    highlight/.style={thick, draw=red!80}
]

\begin{scope}[shift={(0,0)}]
    \node[uv] (u) at (-0.6,0) {$u$};
    \node[uv] (v) at (0.6,0) {$v$};
    
    \node[delta] (x1) at (-1.3, 0.6) {};
    \node[delta] (x2) at (-1.3, -0.6) {};
    
    \node[delta] (y1) at (1.3, 0.6) {};
    \node[delta] (y2) at (1.3, 0) {};
    \node[delta] (y3) at (1.3, -0.6) {};

    \draw[edge] (u) -- (x1); \draw[edge] (u) -- (x2);
    \draw[edge] (v) -- (y1); \draw[edge] (v) -- (y2); \draw[edge] (v) -- (y3);
    
    \node[font=\footnotesize\bfseries] at (0, -1.4) {(a) $d(u)+d(v) \leq 5$};
    \node[font=\scriptsize] at (0, -1.8) {$d(u)=2, d(v)=3 \implies |\Delta_{uv}| = \mathbf{5}$};
\end{scope}

\begin{scope}[shift={(5,0)}]
    \node[uv] (u) at (-0.8,0) {$u$};
    \node[uv] (v) at (0.8,0) {$v$};
    
    \node[shared] (w) at (0, 0.8) {};
    
    \node[delta] (x1) at (-1.5, 0.5) {};
    \node[delta] (x2) at (-1.5, -0.5) {};
    
    \node[delta] (y1) at (1.5, 0.5) {};
    \node[delta] (y2) at (1.5, 0) {};
    \node[delta] (y3) at (1.5, -0.5) {};

    \draw[highlight, dashed] (u) -- (w); \draw[highlight, dashed] (v) -- (w);
    \draw[edge] (u) -- (x1); \draw[edge] (u) -- (x2);
    \draw[edge] (v) -- (y1); \draw[edge] (v) -- (y2); \draw[edge] (v) -- (y3);
    
    \node[font=\footnotesize\bfseries] at (0, -1.4) {(b) dist$(u,v)=2$, sum $\leq 7$};
    \node[font=\scriptsize] at (0, -1.8) {$d(u)=3, d(v)=4 \implies |\Delta_{uv}| = \mathbf{5}$};
\end{scope}

\begin{scope}[shift={(0,-4.2)}]
    \node[uv] (u) at (-0.8,0) {$u$};
    \node[uv] (v) at (0.8,0) {$v$};
    
    \node[shared] (w1) at (0, 0.8) {};
    \node[shared] (w2) at (0, -0.8) {};
    
    \node[delta] (x1) at (-1.6, 0.5) {};
    \node[delta] (x2) at (-1.6, -0.5) {};
    
    \node[delta] (y1) at (1.6, 0.6) {};
    \node[delta] (y2) at (1.6, 0) {};
    \node[delta] (y3) at (1.6, -0.6) {};

    \draw[highlight] (u) -- (v) node[midway, fill=white, inner sep=0.5pt, font=\tiny] {weak};
    \draw[edge] (u) -- (w1); \draw[edge] (v) -- (w1);
    \draw[edge] (u) -- (w2); \draw[edge] (v) -- (w2);
    
    \draw[edge] (u) -- (x1); \draw[edge] (u) -- (x2);
    \draw[edge] (v) -- (y1); \draw[edge] (v) -- (y2); \draw[edge] (v) -- (y3);
    
    \node[font=\footnotesize\bfseries] at (0, -1.6) {(c) Weak edge, weight $\leq 11$};
    \node[font=\scriptsize] at (0, -2.0) {$d(u)=5, d(v)=6 \implies |\Delta_{uv}| = \mathbf{5}$};
\end{scope}

\begin{scope}[shift={(5,-4.2)}]
    \node[uv] (u) at (-0.8,0) {$u$};
    \node[uv] (v) at (0.8,0) {$v$};
    
    \node[shared] (w1) at (0, 0.8) {};
    
    \node[delta] (x1) at (-1.6, 0.5) {};
    \node[delta] (x2) at (-1.6, -0.5) {};
    
    \node[delta] (y1) at (1.6, 0.6) {};
    \node[delta] (y2) at (1.6, 0) {};
    \node[delta] (y3) at (1.6, -0.6) {};

    \draw[highlight] (u) -- (v) node[midway, fill=white, inner sep=0.5pt, font=\tiny] {semiweak};
    \draw[edge] (u) -- (w1); \draw[edge] (v) -- (w1);
    
    \draw[edge] (u) -- (x1); \draw[edge] (u) -- (x2);
    \draw[edge] (v) -- (y1); \draw[edge] (v) -- (y2); \draw[edge] (v) -- (y3);
    
    \node[font=\footnotesize\bfseries] at (0, -1.6) {(d) Semiweak edge, weight $\leq 9$};
    \node[font=\scriptsize] at (0, -2.0) {$d(u)=4, d(v)=5 \implies |\Delta_{uv}| = \mathbf{5}$};
\end{scope}

\begin{scope}[shift={(2.5,-7.2)}]
    \draw[fill=black!5, draw=black!20, rounded corners] (-4,-0.4) rectangle (4, 0.4);
    \node[shared, label={[font=\scriptsize]right:{Common neighbours}}, minimum size=3.5mm] (l2) at (-2.0, 0) {};
    \node[delta, label={[font=\scriptsize]right:{ Not shared}}, minimum size=3.5mm] (l3) at (1.2, 0) {};
\end{scope}
\end{tikzpicture}
\caption{The four cases of Aksionov's theorem, each drawn in the configuration that maximises $|\Delta_{uv}|$. Grey vertices are common neighbours of $u$ and $v$, blue vertices are neighbours of exactly one of them; the latter are precisely the vertices counted by $\Delta_{uv}$. In every case $|\Delta_{uv}| \le 5$, which is \Cref{cor:aksionov}.}
\label{fig:aksionov-cases}
\end{figure}
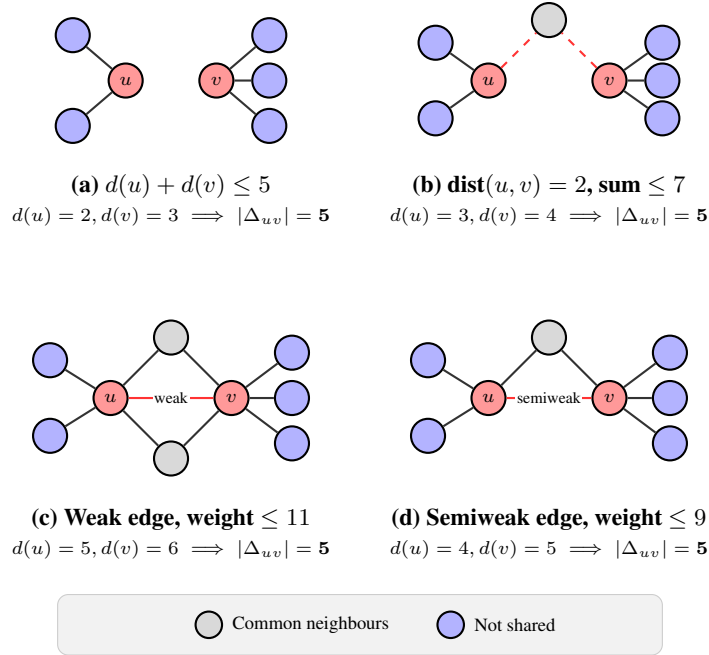

\begin{corollary}\label{cor:aksionov}
Every connected planar graph of order $n \ge 2$ contains two vertices $u,v$ with $|\Delta_{uv}| \le 5$.
\end{corollary}

\begin{proof}
If $u$ and $v$ have $t$ common neighbours, then $|\Delta_{uv}| = \deg(u) + \deg(v) - 2t - 2$ when $u \sim v$, and $|\Delta_{uv}| = \deg(u) + \deg(v) - 2t$ otherwise. A weak edge lies on two triangular faces, so its endpoints have $t \ge 2$; a semiweak edge gives $t \ge 1$; and $\dist(u,v) = 2$ gives $t \ge 1$. The four cases of Aksionov's theorem thus yield, respectively, $|\Delta_{uv}| \le 5$, $\le 7 - 2$, $\le 11 - 4 - 2$, and $\le 9 - 2 - 2$; in every case $|\Delta_{uv}| \le 5$. \Cref{fig:aksionov-cases} illustrates the four cases.
\end{proof}

\begin{observation}\label{t:disconnected_general}
If $F$ is disconnected with components $C_1, \dots, C_c$, then $d(F) \le \min_i d(C_i)$, the minimum taken over components of order at least~$2$: a nontrivial partial automorphism of a component of least depth, extended by the identity on the other components, is a nontrivial partial automorphism of $F$.
\end{observation}

From \Cref{cor:basic-ineq}, \Cref{cor:aksionov} and \Cref{t:disconnected_general} we get the following result.

\begin{theorem}\label{t:planar}
If $F$ is a planar graph, then $d(F)\leq 5$.
\end{theorem}

\section{Tightness of the planar bound via IPR fullerenes}

Given the bounds on planar graphs established in the previous section, a natural class of candidates for extremal asymmetric depth is the class of duals of IPR fullerenes. This class has been extensively studied, efficient generators such as \texttt{buckygen} \cite{buckygen} are available, and as Table~\ref{table:asym_ipr} shows, most IPR fullerenes are asymmetric. The following results explain why (dual) IPR fullerenes are structurally good candidates.

\paragraph{Aksionov-extremality of IPR fullerene duals.} Specialising to triangulations, every face is a triangle, so every edge is weak in the sense of \citet{aksionovDeeplyAsymmetricPlanar2005}, and Aksionov's cases~(a), (b), (d) collapse: there are no semiweak edges, and every vertex has degree at least~$5$ in a fullerene dual, ruling out small-degree-sum pairs. Hence Aksionov's theorem applied to a fullerene dual reduces to the single case~(c): \emph{``there is a weak edge of weight at most~$11$''}. On the other hand, every edge of a fullerene dual has weight in $\{10, 11, 12\}$, and the IPR condition is equivalent to forbidding weight-$10$ edges. Therefore, if $F^* := \dual(F)$ is the dual of an IPR fullerene $F$, then $\min_{uv \in E(F^*)} |\Delta_{uv}| = 5$.

Any non-IPR fullerene dual admits a weight-$10$ edge witnessing $|\Delta| = 4$, hence $d \le 4$ by \Cref{cor:basic-ineq}. Every IPR fullerene dual saturates Aksionov from below, allowing $d = 5$. 

\begin{proposition}\label{thm:dual-depth-5-implies-ipr}
If $F$ is a fullerene graph with $d(\dual(F)) = 5$, then $F$ is IPR.
\end{proposition}

\begin{proof}
If $F$ is not IPR, two pentagonal faces share an edge, so $\dual(F)$ contains adjacent vertices $u \sim v$ of degree $5$ (Figure~\ref{fig:adj_pentagons_fullerene}). Since $\dual(F)$ is a triangulation, the edge $uv$ lies on two triangular faces, so $u$ and $v$ have at least two common neighbours and $|\Delta_{uv}| \le 5 + 5 - 4 - 2 = 4$. By \Cref{cor:basic-ineq}, $d(\dual(F)) \le 4$.
\end{proof}

\begin{figure}[!ht]
\centering
\includegraphics{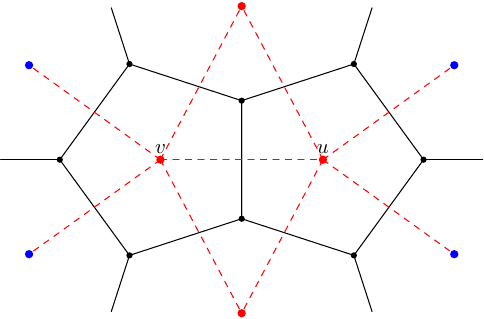}
\caption{Two adjacent pentagons in a fullerene and vertices in the dual.}
\label{fig:adj_pentagons_fullerene}
\end{figure}

It remains to establish that such IPR fullerenes exist. Using \texttt{buckygen}, we generated all nonisomorphic duals of IPR fullerenes, and using \texttt{nauty}, we computed the asymmetric depth for all duals of IPR fullerenes up to $118$ vertices.

\begin{theorem}[Computational]\label{thm:smallest-depth5}
Among the duals of IPR fullerenes (equivalently, $5/6$-regular planar
triangulations with $12$ degree-$5$ vertices, no two adjacent), the
$47$-vertex graphs $\dual(F_1)$ and $\dual(F_2)$---duals of the two
$90$-vertex IPR fullerenes $F_1, F_2$ shown in
\Cref{fig:two_tikz_side_by_side}---are the smallest examples that attain
asymmetric depth $5$.
\end{theorem}

\begin{proof}
The claim was verified by exhaustive enumeration of all IPR fullerenes on up to $118$ vertices (generated with \texttt{buckygen}~\cite{buckygen}) and computation of the asymmetric depth of each dual using \texttt{nauty}. In particular $d(\dual(F)) \le 4$ for each of the $105$ IPR fullerenes on fewer than $90$ vertices, and exactly two of the $46$ IPR fullerenes on $90$ vertices satisfy $d(\dual(F)) = 5$. See \Cref{table:asym_ipr} for the primal depth data and \Cref{sec:computational-methodology} for the methodology.
\end{proof}

\begin{figure}[h]
\centering
\begin{subfigure}[t]{0.48\linewidth}
  \centering
  \begin{tikzpicture}[main_node/.style={circle,fill=black,draw=black,minimum size=3pt,inner sep=0pt}]
    
    \fill[cyan!20] (-1.656340852130327, 1.9082194952535314) -- (-2.0594285714285725, 1.580180597360501) -- (-2.0156390977443603, 0.9779115536003697) -- (-1.589533834586466, 1.010326464459367) -- (-1.3593583959899755, 1.5639731419310028) -- cycle;

    \fill[cyan!20] (-1.559779448621553, 0.0832600138921058) -- (-1.1550075187969935, -0.1883769391062735) -- (-0.8164812030075188, 0.25570733966195913) -- (-0.9607619047619047, 0.7594350544107424) -- (-1.4025864661654142, 0.6356100949293815) -- cycle;

    \fill[cyan!20] (0.3287819548872175, 1.7798564482519106) -- (0.4000802005012529, 2.4093540171335963) -- (-0.08497243107769492, 2.7244269506830285) -- (-0.3488320802005007, 2.2641352164852977) -- (-0.11023558897243158, 1.7422551516554758) -- cycle;

    \fill[cyan!20] (-2.706165413533834, 2.5675387821254914) -- (-3.0559197994987475, 1.799953692984487) -- (-2.7555689223057644, 1.3383653623523963) -- (-2.37606015037594, 1.7688353785598518) -- (-2.2969022556390972, 2.4463070155128497) -- cycle;

    \fill[cyan!20] (-0.9270776942355896, 3.0991433202130128) -- (-1.5277794486215535, 3.2495485065987495) -- (-1.84946365914787, 2.7445241954156057) -- (-1.5339548872180453, 2.398332947441537) -- (-1.0359899749373438, 2.5772632553831905) -- cycle;

    \fill[cyan!20] (-3.169323308270677, -0.5261403102570048) -- (-2.5051829573934836, -1.1887010882148639) -- (-2.443428571428572, -0.5935633248437133) -- (-2.8667268170426063, 0.22912711275758202) -- (-3.351218045112782, 0.468997453114147) -- cycle;

    \fill[cyan!20] (-0.23823558897243124, -0.17865246584857708) -- (0.24793984962406057, 0.10530215327622061) -- (0.2642205513784459, 0.7393378096781653) -- (-0.15570927318295746, 0.806760824264876) -- (-0.433604010025062, 0.2693216022227367) -- cycle;

    \fill[cyan!20] (1.4145363408521296, 0.4547348923361887) -- (1.5307468671679203, 1.4356100949293822) -- (1.105764411027568, 1.8375549895809211) -- (0.878395989974937, 1.1704561241028024) -- (0.9828170426065164, 0.39962954387589633) -- cycle;

    \fill[cyan!20] (-0.07206015037593927, -1.7527205371613794) -- (0.6033082706766919, -1.1621208613104894) -- (0.38941353383458654, -0.7407270201435523) -- (-0.2859548872180451, -0.9961565177124339) -- (-0.6379548872180458, -1.5193331789766162) -- cycle;

    \fill[cyan!20] (-4.9714285714285715, 0.9714285714285715) -- (-2.9717092731829577, -3.0285714285714285) -- (-2.60904260651629, -2.3640657559620273) -- (-3.441042606516291, -0.707663811067377) -- (-4.289323308270676, 0.9435517480898357) -- cycle;

    \fill[cyan!20] (3.0285714285714285, 0.9714285714285715) -- (1.0282907268170414, -3.0285714285714285) -- (0.7262556390977437, -2.3361889326232905) -- (1.5273784461152884, -0.6823801805973604) -- (2.3537644110275675, 0.9727251678629307) -- cycle;

    \fill[cyan!20] (1.0282907268170414, 4.9714285714285715) -- (-2.9717092731829577, 4.9714285714285715) -- (-2.6309373433583962, 4.302384811298912) -- (-0.9708671679198004, 4.295901829127113) -- (0.6746065162907255, 4.307571197036351) -- cycle;

    \node[main_node] (0) at (3.0285714285714285, 0.9714285714285715) {};
    \node[main_node] (1) at (1.0282907268170414, -3.0285714285714285) {};
    \node[main_node] (2) at (0.7262556390977437, -2.3361889326232905) {};
    \node[main_node] (3) at (1.5273784461152884, -0.6823801805973604) {};
    \node[main_node] (4) at (2.3537644110275675, 0.9727251678629307) {};
    \node[main_node] (5) at (1.0282907268170414, 4.9714285714285715) {};
    \node[main_node] (6) at (-2.9717092731829577, 4.9714285714285715) {};
    \node[main_node] (7) at (-4.9714285714285715, 0.9714285714285715) {};
    \node[main_node] (8) at (-2.9717092731829577, -3.0285714285714285) {};
    \node[main_node] (9) at (-2.60904260651629, -2.3640657559620273) {};
    \node[main_node] (10) at (-1.5900952380952385, -2.275248900208382) {};
    \node[main_node] (11) at (-0.3022355889724313, -2.2661727251678627) {};
    \node[main_node] (12) at (-0.07206015037593927, -1.7527205371613794) {};
    \node[main_node] (13) at (0.6033082706766919, -1.1621208613104894) {};
    \node[main_node] (14) at (1.2500451127819545, -0.5047464690900663) {};
    \node[main_node] (15) at (1.4145363408521296, 0.4547348923361887) {};
    \node[main_node] (16) at (1.5307468671679203, 1.4356100949293822) {};
    \node[main_node] (17) at (1.7878696741854627, 1.9425793007640655) {};
    \node[main_node] (18) at (1.1299047619047622, 3.2534382959018298) {};
    \node[main_node] (19) at (0.6746065162907255, 4.307571197036351) {};
    \node[main_node] (20) at (-0.9708671679198004, 4.295901829127113) {};
    \node[main_node] (21) at (-2.6309373433583962, 4.302384811298912) {};
    \node[main_node] (22) at (-3.0592882205513785, 3.2495485065987495) {};
    \node[main_node] (23) at (-3.720060150375941, 1.9244269506830292) {};
    \node[main_node] (24) at (-4.289323308270676, 0.9435517480898357) {};
    \node[main_node] (25) at (-3.441042606516291, -0.707663811067377) {};
    \node[main_node] (26) at (-3.169323308270677, -0.5261403102570048) {};
    \node[main_node] (27) at (-2.5051829573934836, -1.1887010882148639) {};
    \node[main_node] (28) at (-1.8438496240601498, -1.7429960639036803) {};
    \node[main_node] (29) at (-1.338586466165415, -1.4927529520722391) {};
    \node[main_node] (30) at (-0.6379548872180458, -1.5193331789766162) {};
    \node[main_node] (31) at (-0.2859548872180451, -0.9961565177124339) {};
    \node[main_node] (32) at (0.38941353383458654, -0.7407270201435523) {};
    \node[main_node] (33) at (0.575799498746866, -0.11252604769622554) {};
    \node[main_node] (34) at (0.9828170426065164, 0.39962954387589633) {};
    \node[main_node] (35) at (0.878395989974937, 1.1704561241028024) {};
    \node[main_node] (36) at (1.105764411027568, 1.8375549895809211) {};
    \node[main_node] (37) at (0.7722907268170411, 2.5396619587867564) {};
    \node[main_node] (38) at (0.6313784461152885, 3.226209770780274) {};
    \node[main_node] (39) at (-0.11697243107769406, 3.5996295438758974) {};
    \node[main_node] (40) at (-0.9209022556390978, 3.9179439685112296) {};
    \node[main_node] (41) at (-1.7804110275689222, 3.613243806436676) {};
    \node[main_node] (42) at (-2.5607619047619057, 3.233989349386432) {};
    \node[main_node] (43) at (-2.706165413533834, 2.5675387821254914) {};
    \node[main_node] (44) at (-3.0559197994987475, 1.799953692984487) {};
    \node[main_node] (45) at (-3.4477794486215543, 1.4155128501968048) {};
    \node[main_node] (46) at (-3.351218045112782, 0.468997453114147) {};
    \node[main_node] (47) at (-2.8667268170426063, 0.22912711275758202) {};
    \node[main_node] (48) at (-2.443428571428572, -0.5935633248437133) {};
    \node[main_node] (49) at (-1.9943057644110282, -0.5955082194952546) {};
    \node[main_node] (50) at (-1.5013934837092737, -0.9786524658485778) {};
    \node[main_node] (51) at (-1.0730426065162906, -0.6694142162537631) {};
    \node[main_node] (52) at (-0.5363408521303263, -0.6616346376476026) {};
    \node[main_node] (53) at (-0.23823558897243124, -0.17865246584857708) {};
    \node[main_node] (54) at (0.24793984962406057, 0.10530215327622061) {};
    \node[main_node] (55) at (0.2642205513784459, 0.7393378096781653) {};
    \node[main_node] (56) at (0.5207819548872186, 1.2365825422551515) {};
    \node[main_node] (57) at (0.3287819548872175, 1.7798564482519106) {};
    \node[main_node] (58) at (0.4000802005012529, 2.4093540171335963) {};
    \node[main_node] (59) at (-0.08497243107769492, 2.7244269506830285) {};
    \node[main_node] (60) at (-0.362305764411027, 3.2132438064366755) {};
    \node[main_node] (61) at (-0.9270776942355896, 3.0991433202130128) {};
    \node[main_node] (62) at (-1.5277794486215535, 3.2495485065987495) {};
    \node[main_node] (63) at (-1.84946365914787, 2.7445241954156057) {};
    \node[main_node] (64) at (-2.2969022556390972, 2.4463070155128497) {};
    \node[main_node] (65) at (-2.37606015037594, 1.7688353785598518) {};
    \node[main_node] (66) at (-2.7555689223057644, 1.3383653623523963) {};
    \node[main_node] (67) at (-2.6573233082706773, 0.6738596897429967) {};
    \node[main_node] (68) at (-2.2486215538847123, 0.4988191711044223) {};
    \node[main_node] (69) at (-1.9673583959899754, -0.04250984024079685) {};
    \node[main_node] (70) at (-1.559779448621553, 0.0832600138921058) {};
    \node[main_node] (71) at (-1.1550075187969935, -0.1883769391062735) {};
    \node[main_node] (72) at (-0.8164812030075188, 0.25570733966195913) {};
    \node[main_node] (73) at (-0.433604010025062, 0.2693216022227367) {};
    \node[main_node] (74) at (-0.15570927318295746, 0.806760824264876) {};
    \node[main_node] (75) at (-0.3258145363408529, 1.2748321370687652) {};
    \node[main_node] (76) at (-0.11023558897243158, 1.7422551516554758) {};
    \node[main_node] (77) at (-0.3488320802005007, 2.2641352164852977) {};
    \node[main_node] (78) at (-0.7738145363408511, 2.205788376939106) {};
    \node[main_node] (79) at (-1.0359899749373438, 2.5772632553831905) {};
    \node[main_node] (80) at (-1.5339548872180453, 2.398332947441537) {};
    \node[main_node] (81) at (-1.656340852130327, 1.9082194952535314) {};
    \node[main_node] (82) at (-2.0594285714285725, 1.580180597360501) {};
    \node[main_node] (83) at (-2.0156390977443603, 0.9779115536003697) {};
    \node[main_node] (84) at (-1.589533834586466, 1.010326464459367) {};
    \node[main_node] (85) at (-1.4025864661654142, 0.6356100949293815) {};
    \node[main_node] (86) at (-0.9607619047619047, 0.7594350544107424) {};
    \node[main_node] (87) at (-0.736200501253133, 1.2553831905533694) {};
    \node[main_node] (88) at (-0.9478496240601513, 1.6877981014123646) {};
    \node[main_node] (89) at (-1.3593583959899755, 1.5639731419310028) {};

    \path[draw, thick]
    (0) edge node {} (1)
    (0) edge node {} (4)
    (0) edge node {} (5)
    (1) edge node {} (2)
    (1) edge node {} (8)
    (2) edge node {} (3)
    (2) edge node {} (11)
    (3) edge node {} (4)
    (3) edge node {} (14)
    (4) edge node {} (17)
    (5) edge node {} (6)
    (5) edge node {} (19)
    (6) edge node {} (7)
    (6) edge node {} (21)
    (7) edge node {} (8)
    (7) edge node {} (24)
    (8) edge node {} (9)
    (9) edge node {} (10)
    (9) edge node {} (25)
    (10) edge node {} (11)
    (10) edge node {} (28)
    (11) edge node {} (12)
    (12) edge node {} (13)
    (12) edge node {} (30)
    (13) edge node {} (14)
    (13) edge node {} (32)
    (14) edge node {} (15)
    (15) edge node {} (16)
    (15) edge node {} (34)
    (16) edge node {} (17)
    (16) edge node {} (36)
    (17) edge node {} (18)
    (18) edge node {} (19)
    (18) edge node {} (38)
    (19) edge node {} (20)
    (20) edge node {} (21)
    (20) edge node {} (40)
    (21) edge node {} (22)
    (22) edge node {} (23)
    (22) edge node {} (42)
    (23) edge node {} (24)
    (23) edge node {} (45)
    (24) edge node {} (25)
    (25) edge node {} (26)
    (26) edge node {} (27)
    (26) edge node {} (46)
    (27) edge node {} (28)
    (27) edge node {} (48)
    (28) edge node {} (29)
    (29) edge node {} (30)
    (29) edge node {} (50)
    (30) edge node {} (31)
    (31) edge node {} (32)
    (31) edge node {} (52)
    (32) edge node {} (33)
    (33) edge node {} (34)
    (33) edge node {} (54)
    (34) edge node {} (35)
    (35) edge node {} (36)
    (35) edge node {} (56)
    (36) edge node {} (37)
    (37) edge node {} (38)
    (37) edge node {} (58)
    (38) edge node {} (39)
    (39) edge node {} (40)
    (39) edge node {} (60)
    (40) edge node {} (41)
    (41) edge node {} (42)
    (41) edge node {} (62)
    (42) edge node {} (43)
    (43) edge node {} (44)
    (43) edge node {} (64)
    (44) edge node {} (45)
    (44) edge node {} (66)
    (45) edge node {} (46)
    (46) edge node {} (47)
    (47) edge node {} (48)
    (47) edge node {} (67)
    (48) edge node {} (49)
    (49) edge node {} (50)
    (49) edge node {} (69)
    (50) edge node {} (51)
    (51) edge node {} (52)
    (51) edge node {} (71)
    (52) edge node {} (53)
    (53) edge node {} (54)
    (53) edge node {} (73)
    (54) edge node {} (55)
    (55) edge node {} (56)
    (55) edge node {} (74)
    (56) edge node {} (57)
    (57) edge node {} (58)
    (57) edge node {} (76)
    (58) edge node {} (59)
    (59) edge node {} (60)
    (59) edge node {} (77)
    (60) edge node {} (61)
    (61) edge node {} (62)
    (61) edge node {} (79)
    (62) edge node {} (63)
    (63) edge node {} (64)
    (63) edge node {} (80)
    (64) edge node {} (65)
    (65) edge node {} (66)
    (65) edge node {} (82)
    (66) edge node {} (67)
    (67) edge node {} (68)
    (68) edge node {} (69)
    (68) edge node {} (83)
    (69) edge node {} (70)
    (70) edge node {} (71)
    (70) edge node {} (85)
    (71) edge node {} (72)
    (72) edge node {} (73)
    (72) edge node {} (86)
    (73) edge node {} (74)
    (74) edge node {} (75)
    (75) edge node {} (76)
    (75) edge node {} (87)
    (76) edge node {} (77)
    (77) edge node {} (78)
    (78) edge node {} (79)
    (78) edge node {} (88)
    (79) edge node {} (80)
    (80) edge node {} (81)
    (81) edge node {} (82)
    (81) edge node {} (89)
    (82) edge node {} (83)
    (83) edge node {} (84)
    (84) edge node {} (85)
    (84) edge node {} (89)
    (85) edge node {} (86)
    (86) edge node {} (87)
    (87) edge node {} (88)
    (88) edge node {} (89)
    ;

\end{tikzpicture}
\end{subfigure}\hfill
\begin{subfigure}[t]{0.48\linewidth}
  \centering
 
\begin{tikzpicture}[main_node/.style={circle,fill=black,draw=black,minimum size=3pt,inner sep=0pt}]
    
    \fill[cyan!20] (0.8039736952567509, -0.1422916690708922) -- (0.8218833076815457, 0.603487473024594) -- (0.3567930600251854, 0.5485557337887892) -- (0.09878270603050154, -0.11450267157513316) -- (0.31537708129285047, -0.6063033016744948) -- cycle;

    \fill[cyan!20] (-1.4442423394431225, 1.9250804934624313) -- (-1.8903036238981388, 2.155147542497085) -- (-2.240660416958163, 1.6814420735577529) -- (-2.0044774031061987, 1.2394031366484715) -- (-1.4934937736113065, 1.3990283083101565) -- cycle;

    \fill[cyan!20] (-3.368406324331888, 1.599367592581907) -- (-2.8439904855183986, 2.085998176635546) -- (-2.6587379319994398, 1.592905035024753) -- (-2.9156289352175735, 0.902057632165072) -- (-3.3572128165663906, 0.7786227828234447) -- cycle;

    \fill[cyan!20] (-1.7610186092066602, -0.030489423332141996) -- (-1.7694137400307817, 0.48586892548440375) -- (-1.2668252413600103, 0.6532491662146729) -- (-0.9696376101860911, 0.24675429586973419) -- (-1.2987267384916743, -0.1836520374366728) -- cycle;

    \fill[cyan!20] (-1.0351196306142434, 2.078889363322678) -- (-0.937176437666154, 2.5726287606891862) -- (-0.3584720861900106, 2.458887747683287) -- (-0.16706310340002783, 1.8979377517223872) -- (-0.6103260109136706, 1.6891971426263377) -- cycle;

    \fill[cyan!20] (-2.8551839932838945, -0.39885520408987496) -- (-2.5680705190989226, 0.0005308529421954589) -- (-2.095144816006716, -0.2928692601525622) -- (-1.9708968798097102, -0.8596355579149044) -- (-2.4180775150412757, -1.0308933331794634) -- cycle;

    \fill[cyan!20] (-1.1336224989506096, -0.5979019768501947) -- (-1.4716664334685878, -1.0166757065537269) -- (-1.1576885406464257, -1.4412657380586946) -- (-0.5442843150972432, -1.1769471339711244) -- (-0.613124387855045, -0.6424936239945538) -- cycle;

    \fill[cyan!20] (1.100041975654121, -0.3497397666555111) -- (0.8291590877291153, -1.2861643566870162) -- (1.2796977752903318, -1.3902115333571816) -- (1.9149293409822319, -0.07508107047649837) -- (1.5998320973835165, 0.3928080966613976) -- cycle;

    \fill[cyan!20] (-0.39485098642787086, 3.3197004142961006) -- (0.15866797257590726, 3.5168084197892746) -- (0.7894221351616064, 2.8627975950053663) -- (0.5487617182034423, 2.4220511696075153) -- (-0.07695536588778396, 2.746471558976607) -- cycle;

    \fill[cyan!20] (-2.497551420176297, -2.496601387141819) -- (-0.7877431089967835, -2.4759212029589275) -- (0.7961382398209054, -2.481091249004651) -- (1.1431369805512779, -3.1428571428571423) -- (-2.857422694836994, -3.1428571428571423) -- cycle;

    \fill[cyan!20] (0.8459493493773618, 4.158540385214591) -- (1.1431369805512779, 4.857142857142858) -- (3.1428571428571423, 0.8574659850207142) -- (2.475164404645305, 0.8561734735092843) -- (1.669791520917868, 2.47375163006474) -- cycle;

    \fill[cyan!20] (-2.391213096404085, 3.2350409102973923) -- (-3.020847908213236, 3.12905496636008) -- (-2.535609346578984, 4.226397239564701) -- (-1.2875332307261775, 4.171465500328898) -- (-1.5987127466069682, 3.681603637496682) -- cycle;

    \node[main_node] (0) at (0.8039736952567509, -0.1422916690708922) {};
    \node[main_node] (1) at (0.8218833076815457, 0.603487473024594) {};
    \node[main_node] (2) at (1.100041975654121, -0.3497397666555111) {};
    \node[main_node] (3) at (0.31537708129285047, -0.6063033016744948) {};
    \node[main_node] (4) at (-3.562054008674968, -0.39691643682273003) {};
    \node[main_node] (5) at (-4.216314537568211, 0.8296769875249552) {};
    \node[main_node] (6) at (-3.3297887225409255, -0.022088098507840215) {};
    \node[main_node] (7) at (-3.020847908213236, 3.12905496636008) {};
    \node[main_node] (8) at (-3.368406324331888, 1.599367592581907) {};
    \node[main_node] (9) at (-1.4442423394431225, 1.9250804934624313) {};
    \node[main_node] (10) at (-1.8903036238981388, 2.155147542497085) {};
    \node[main_node] (11) at (-1.332866937176438, 2.93000819359976) {};
    \node[main_node] (12) at (-1.8964600531691609, 2.650179451375024) {};
    \node[main_node] (13) at (-2.413040436546802, 2.7128662596794104) {};
    \node[main_node] (14) at (-2.8439904855183986, 2.085998176635546) {};
    \node[main_node] (15) at (-2.391213096404085, 3.2350409102973923) {};
    \node[main_node] (16) at (-2.535609346578984, 4.226397239564701) {};
    \node[main_node] (17) at (-1.2875332307261775, 4.171465500328898) {};
    \node[main_node] (18) at (-2.240660416958163, 1.6814420735577529) {};
    \node[main_node] (19) at (-2.9156289352175735, 0.902057632165072) {};
    \node[main_node] (20) at (-1.5987127466069682, 3.681603637496682) {};
    \node[main_node] (21) at (-2.6587379319994398, 1.592905035024753) {};
    \node[main_node] (22) at (-2.0044774031061987, 1.2394031366484715) {};
    \node[main_node] (23) at (-2.5747866237582206, 0.5569570586130892) {};
    \node[main_node] (24) at (-3.3572128165663906, 0.7786227828234447) {};
    \node[main_node] (25) at (-2.8551839932838945, -0.39885520408987496) {};
    \node[main_node] (26) at (-3.6359311599272424, 1.9192641916609923) {};
    \node[main_node] (27) at (-2.5680705190989226, 0.0005308529421954589) {};
    \node[main_node] (28) at (-2.1404785224569745, 0.7611738774191306) {};
    \node[main_node] (29) at (-1.7610186092066602, -0.030489423332141996) {};
    \node[main_node] (30) at (-1.7694137400307817, 0.48586892548440375) {};
    \node[main_node] (31) at (-2.095144816006716, -0.2928692601525622) {};
    \node[main_node] (32) at (-1.4934937736113065, 1.3990283083101565) {};
    \node[main_node] (33) at (-1.1000419756541204, 3.3882035244019235) {};
    \node[main_node] (34) at (-2.857422694836994, 4.857142857142858) {};
    \node[main_node] (35) at (-1.0351196306142434, 2.078889363322678) {};
    \node[main_node] (36) at (-0.937176437666154, 2.5726287606891862) {};
    \node[main_node] (37) at (-4.857142857142858, 0.8574659850207142) {};
    \node[main_node] (38) at (-0.03330068560234967, 4.060955766101577) {};
    \node[main_node] (39) at (-2.4180775150412757, -1.0308933331794634) {};
    \node[main_node] (40) at (-2.833916328529452, -1.6131697690789695) {};
    \node[main_node] (41) at (-0.16706310340002783, 1.8979377517223872) {};
    \node[main_node] (42) at (-1.9708968798097102, -0.8596355579149044) {};
    \node[main_node] (43) at (-0.3584720861900106, 2.458887747683287) {};
    \node[main_node] (44) at (-0.6103260109136706, 1.6891971426263377) {};
    \node[main_node] (45) at (-2.2664054848188053, -1.5892583061175023) {};
    \node[main_node] (46) at (-1.1336224989506096, -0.5979019768501947) {};
    \node[main_node] (47) at (-1.2987267384916743, -0.1836520374366728) {};
    \node[main_node] (48) at (-1.1353015251154333, 1.0894718013225149) {};
    \node[main_node] (49) at (-0.39485098642787086, 3.3197004142961006) {};
    \node[main_node] (50) at (0.15866797257590726, 3.5168084197892746) {};
    \node[main_node] (51) at (-2.497551420176297, -2.496601387141819) {};
    \node[main_node] (52) at (-2.857422694836994, -3.1428571428571423) {};
    \node[main_node] (53) at (0.8459493493773618, 4.158540385214591) {};
    \node[main_node] (54) at (1.1431369805512779, 4.857142857142858) {};
    \node[main_node] (55) at (-1.2668252413600103, 0.6532491662146729) {};
    \node[main_node] (56) at (-0.7071498530852125, -2.0894602610411637) {};
    \node[main_node] (57) at (1.3921925283335668, 2.237222023472932) {};
    \node[main_node] (58) at (-1.4610326010913663, -1.8354817490450417) {};
    \node[main_node] (59) at (0.7894221351616064, 2.8627975950053663) {};
    \node[main_node] (60) at (-0.7877431089967835, -2.4759212029589275) {};
    \node[main_node] (61) at (1.669791520917868, 2.47375163006474) {};
    \node[main_node] (62) at (-1.4716664334685878, -1.0166757065537269) {};
    \node[main_node] (63) at (-0.9696376101860911, 0.24675429586973419) {};
    \node[main_node] (64) at (-0.6718903036238997, 1.20644409310699) {};
    \node[main_node] (65) at (-0.07695536588778396, 2.746471558976607) {};
    \node[main_node] (66) at (3.1428571428571423, 0.8574659850207142) {};
    \node[main_node] (67) at (1.1431369805512779, -3.1428571428571423) {};
    \node[main_node] (68) at (-1.1576885406464257, -1.4412657380586946) {};
    \node[main_node] (69) at (0.5487617182034423, 2.4220511696075153) {};
    \node[main_node] (70) at (0.7961382398209054, -2.481091249004651) {};
    \node[main_node] (71) at (2.475164404645305, 0.8561734735092843) {};
    \node[main_node] (72) at (-0.5129424933538544, 0.29651598905981214) {};
    \node[main_node] (73) at (-0.32768993983489736, 0.8322620105478169) {};
    \node[main_node] (74) at (0.24149993004057535, 1.5554222011932648) {};
    \node[main_node] (75) at (-0.613124387855045, -0.6424936239945538) {};
    \node[main_node] (76) at (-0.5442843150972432, -1.1769471339711244) {};
    \node[main_node] (77) at (-0.31761578284594894, -0.17589696836808866) {};
    \node[main_node] (78) at (0.12340842311459355, 0.9905946706980724) {};
    \node[main_node] (79) at (0.660137120470127, 1.7635165545335993) {};
    \node[main_node] (80) at (1.9149293409822319, -0.07508107047649837) {};
    \node[main_node] (81) at (1.2796977752903318, -1.3902115333571816) {};
    \node[main_node] (82) at (1.4940534489995796, 1.2930423643728426) {};
    \node[main_node] (83) at (0.12340842311459355, -1.6945979942991016) {};
    \node[main_node] (84) at (1.0507905414859389, 1.2607295765870772) {};
    \node[main_node] (85) at (0.8291590877291153, -1.2861643566870162) {};
    \node[main_node] (86) at (-0.024345879389954916, -1.2389876865197982) {};
    \node[main_node] (87) at (1.5998320973835165, 0.3928080966613976) {};
    \node[main_node] (88) at (0.09878270603050154, -0.11450267157513316) {};
    \node[main_node] (89) at (0.3567930600251854, 0.5485557337887892) {};

    \path[draw, thick]
    (0) edge node {} (1)
    (0) edge node {} (2)
    (0) edge node {} (3)
    (1) edge node {} (84)
    (1) edge node {} (89)
    (2) edge node {} (85)
    (2) edge node {} (87)
    (3) edge node {} (86)
    (3) edge node {} (88)
    (4) edge node {} (5)
    (4) edge node {} (6)
    (4) edge node {} (40)
    (5) edge node {} (26)
    (5) edge node {} (37)
    (6) edge node {} (24)
    (6) edge node {} (25)
    (7) edge node {} (15)
    (7) edge node {} (16)
    (7) edge node {} (26)
    (8) edge node {} (14)
    (8) edge node {} (24)
    (8) edge node {} (26)
    (9) edge node {} (10)
    (9) edge node {} (32)
    (9) edge node {} (35)
    (10) edge node {} (12)
    (10) edge node {} (18)
    (11) edge node {} (12)
    (11) edge node {} (33)
    (11) edge node {} (36)
    (12) edge node {} (13)
    (13) edge node {} (14)
    (13) edge node {} (15)
    (14) edge node {} (21)
    (15) edge node {} (20)
    (16) edge node {} (17)
    (16) edge node {} (34)
    (17) edge node {} (20)
    (17) edge node {} (38)
    (18) edge node {} (21)
    (18) edge node {} (22)
    (19) edge node {} (21)
    (19) edge node {} (23)
    (19) edge node {} (24)
    (20) edge node {} (33)
    (22) edge node {} (28)
    (22) edge node {} (32)
    (23) edge node {} (27)
    (23) edge node {} (28)
    (25) edge node {} (27)
    (25) edge node {} (39)
    (27) edge node {} (31)
    (28) edge node {} (30)
    (29) edge node {} (30)
    (29) edge node {} (31)
    (29) edge node {} (47)
    (30) edge node {} (55)
    (31) edge node {} (42)
    (32) edge node {} (48)
    (33) edge node {} (49)
    (34) edge node {} (37)
    (34) edge node {} (54)
    (35) edge node {} (36)
    (35) edge node {} (44)
    (36) edge node {} (43)
    (37) edge node {} (52)
    (38) edge node {} (50)
    (38) edge node {} (53)
    (39) edge node {} (42)
    (39) edge node {} (45)
    (40) edge node {} (45)
    (40) edge node {} (51)
    (41) edge node {} (43)
    (41) edge node {} (44)
    (41) edge node {} (74)
    (42) edge node {} (62)
    (43) edge node {} (65)
    (44) edge node {} (64)
    (45) edge node {} (58)
    (46) edge node {} (47)
    (46) edge node {} (62)
    (46) edge node {} (75)
    (47) edge node {} (63)
    (48) edge node {} (55)
    (48) edge node {} (64)
    (49) edge node {} (50)
    (49) edge node {} (65)
    (50) edge node {} (59)
    (51) edge node {} (52)
    (51) edge node {} (60)
    (52) edge node {} (67)
    (53) edge node {} (54)
    (53) edge node {} (61)
    (54) edge node {} (66)
    (55) edge node {} (63)
    (56) edge node {} (58)
    (56) edge node {} (60)
    (56) edge node {} (83)
    (57) edge node {} (59)
    (57) edge node {} (61)
    (57) edge node {} (82)
    (58) edge node {} (68)
    (59) edge node {} (69)
    (60) edge node {} (70)
    (61) edge node {} (71)
    (62) edge node {} (68)
    (63) edge node {} (72)
    (64) edge node {} (73)
    (65) edge node {} (69)
    (66) edge node {} (67)
    (66) edge node {} (71)
    (67) edge node {} (70)
    (68) edge node {} (76)
    (69) edge node {} (79)
    (70) edge node {} (81)
    (71) edge node {} (80)
    (72) edge node {} (73)
    (72) edge node {} (77)
    (73) edge node {} (78)
    (74) edge node {} (78)
    (74) edge node {} (79)
    (75) edge node {} (76)
    (75) edge node {} (77)
    (76) edge node {} (86)
    (77) edge node {} (88)
    (78) edge node {} (89)
    (79) edge node {} (84)
    (80) edge node {} (81)
    (80) edge node {} (87)
    (81) edge node {} (85)
    (82) edge node {} (84)
    (82) edge node {} (87)
    (83) edge node {} (85)
    (83) edge node {} (86)
    (88) edge node {} (89);

\end{tikzpicture}
\end{subfigure}
\caption{The smallest fullerenes whose duals attain maximal asymmetric depth for planar graphs. They have 90 vertices, and their duals have 47 vertices.}
\label{fig:two_tikz_side_by_side}

\end{figure}
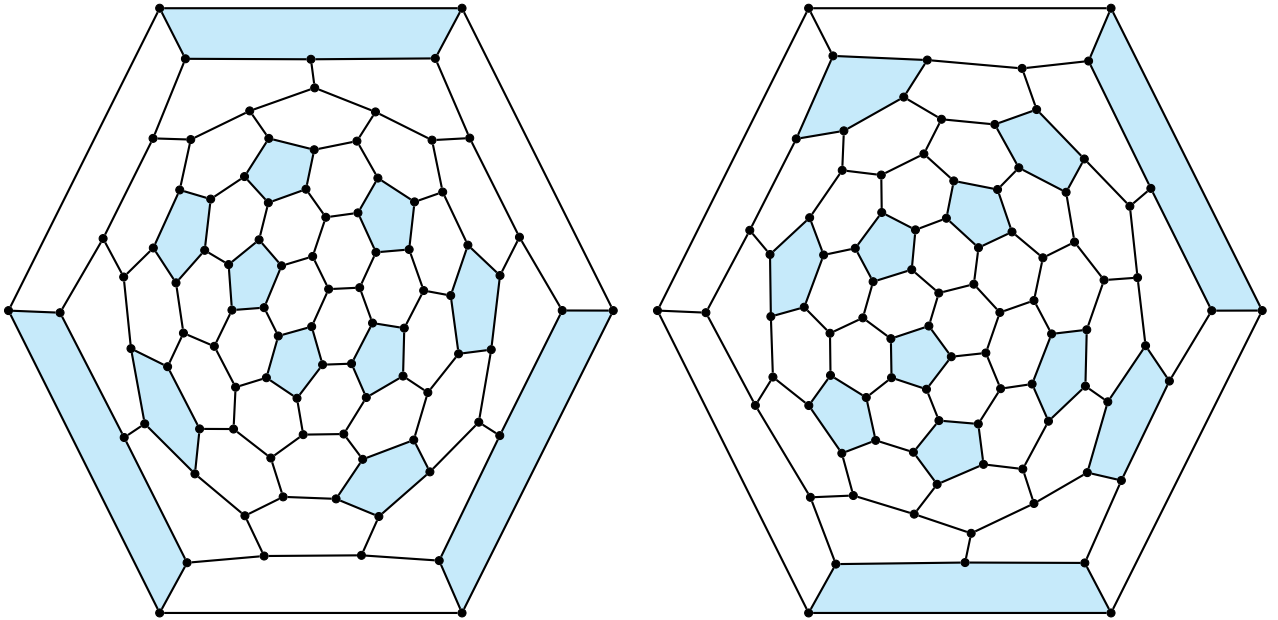

\section{Hidden near-mirror symmetries of low-depth IPR fullerenes}\label{sec:hidden}

In this section, we examine fullerenes of low asymmetric depth, as their partial-automorphism structure is constrained in striking ways. The following lemma is a key tool in our analysis, and it relies on the fact that fullerenes are regular graphs. It states that any isomorphism between vertex-deleted subgraphs of a cubic graph can be uniquely extended to an automorphism of the entire graph. 

\begin{lemma} \label{lem:unique_extension}
Let $F$ be a cubic graph, and let $u, v \in V(F)$ be distinct vertices. Any graph isomorphism $\varphi: F-u \to F-v$ extends uniquely to an automorphism $\Phi \in \Aut(F)$ such that $\Phi(u) = v$.
\end{lemma}

\begin{proof}
Since $F$ is cubic, $F-u$ and $F-v$ each have exactly three vertices of degree~$2$, namely $N_F(u)$ and $N_F(v)$; any isomorphism $\varphi \colon F-u \to F-v$ preserves degrees and so maps $N_F(u)$ onto $N_F(v)$. Extend it by $\Phi(u) = v$. Edges avoiding $u$ are preserved because $\varphi$ is an isomorphism, and each edge $uw$ maps to $v\,\varphi(w)$ with $\varphi(w) \in N_F(v)$, again an edge; so $\Phi$ is an automorphism, and it is the only one with $\Phi(u)=v$ extending $\varphi$.
\end{proof}

\begin{theorem}\label{thm:fullerene-depth-gt-1}
Let $F$ be a fullerene graph of order $n$. Then, $d(F)\in \{0,2,3,4\}$.
\end{theorem}

\begin{proof}
Suppose $d(F)=1$. Then $F$ admits a non-trivial partial automorphism $\varphi$ of rank $n-1$; in particular, $\Aut(F)=\{\id\}$, i.e., $F$ is asymmetric. Let the domain of $\varphi$ be $F-u$ and the range be $F-v$. 

If $u = v$, then $\varphi$ is a non-identity automorphism of $F-u$. By mapping $u \mapsto u$, this extends uniquely to a non-identity automorphism of $F$, contradicting $\Aut(F) = \{\id\}$. 
If $u \neq v$, by Lemma~\ref{lem:unique_extension}, the isomorphism $\varphi\colon F-u\to F-v$ extends uniquely to an automorphism $\Phi\in\Aut(F)$ with $\Phi(u)=v$. Since $\Aut(F)$ is trivial, $\Phi=\id$, forcing $u=v$---a contradiction.

Since $F$ is cubic, by Lemma~\ref{lem:swap} there is a partial automorphism of rank $n-4$ swapping two vertices; therefore, $d(F) \le 4$. Consequently, $d(F) \in \{0,2,3,4\}$.
\end{proof}

All four values permitted by \Cref{thm:fullerene-depth-gt-1} are realised by IPR fullerenes in the enumerated range, though not simultaneously at every size: at $n = 84$ the unique asymmetric IPR fullerene has $d(F) = 3$, the first depth-$4$ examples appear at $n = 86$, and the first depth-$2$ example at $n = 92$ (\Cref{table:asym_ipr}). The depth-$2$ and depth-$3$ column counts are not monotone in $n$: they correspond to structurally specific configurations (cf.\ \Cref{cor:hidden-symmetry-primal}) that do not appear in every size class.

\subsection{Localised partial automorphisms in the fullerene}\label{subsec:localised}

Fullerenes are cyclically 5-edge-connected, so they contain no cyclic $3$- or $4$-edge cuts. Furthermore, Kardo\v{s} and \v{S}krekovski (Corollary~6 in \cite{cyclic_edge_cuts}) proved that a fullerene graph contains a nontrivial cyclic $5$- or $6$-edge cut if and only if it contains a pair of adjacent pentagons. Since this paper considers only IPR fullerenes, our fullerene graphs have no non-trivial cyclic $5$- or $6$-edge cuts; in particular, every non-trivial cyclic edge cut in an IPR fullerene has size at least~$7$.  

Any largest nontrivial partial automorphism that realises a low asymmetric depth in an asymmetric IPR fullerene must spread its support across at least two distinct faces of the fullerene.

This structural property of IPR fullerenes implies that any partial automorphism with support contained within a single face must be trivial, as we now show.

\begin{definition}[Face-localised partial automorphism]\label{def:face-localised}
Let $F$ be a fullerene graph and $\varphi$ a partial automorphism of $F$ with support $\supp(\varphi)$.  We say $\varphi$ is \emph{face-localised} if there exists a face $F^\circ$ of $F$ such that $\supp(\varphi) \subseteq V(F^\circ)$.
\end{definition}

Each $v\in V(F^\circ)$ has degree $3$ with two neighbours on $\partial F^\circ$, so its
third neighbour $\ex(v)\notin V(F^\circ)$ is well defined.

\begin{observation}[Distinct exterior neighbours]\label{lem:distinct-exterior}
Let $F$ be an IPR fullerene and $F^\circ$ a face of $F$.  For every $v \in V(F^\circ)$ let $\ex(v) \in V(F) \setminus V(F^\circ)$ denote the unique neighbour of $v$ outside $F^\circ$.  Then the map $v \mapsto \ex(v)$ is injective; equivalently, the $|V(F^\circ)|$ exterior neighbours are pairwise distinct.
\end{observation}

Indeed, a common exterior neighbour of two vertices of $F^\circ$ at boundary distance $\ell$ would close a cycle of length $\ell + 2$, so girth $5$ forces $\ell \ge 3$ and hence $|V(F^\circ)| = 6$ and $\ell = 3$; the resulting $5$-cycle bounds a pentagonal face (see the facts collected before the proof of \Cref{thm:unified-localisation}), which would then share three edges with $F^\circ$, contradicting \Cref{lem:three-faces}\ref{it:faces-a}.

The next two observations isolate the girth and planarity arguments used repeatedly below.

\begin{observation}[Anchors]\label{lem:anchor}
Let $F$ be a graph of girth at least $5$, let $\varphi$ be a partial automorphism of $F$, and let $u \in \dom(\varphi)$ with $\varphi(u) \neq u$. Then every neighbour of $u$ fixed by $\varphi$ is a common neighbour of $u$ and $\varphi(u)$. Consequently, $u$ has at most one fixed neighbour, and none at all if $u$ is adjacent to $\varphi(u)$. The same statements hold with $\varphi^{-1}$ in place of $\varphi$.
\end{observation}

\begin{proof}
If $w \in \fix(\varphi)$ and $uw \in E(F)$, then $\varphi$ maps the edge $uw$ to the edge $\varphi(u)\,w$, so $w \in N(u) \cap N(\varphi(u))$. In a graph of girth at least $5$, two distinct vertices have at most one common neighbour (two would close a $4$-cycle), and adjacent vertices have none (a triangle). Since $\varphi^{-1}$ is a partial automorphism with $\fix(\varphi^{-1}) = \fix(\varphi)$, the last claim follows.
\end{proof}

\begin{observation}[Pairwise meeting faces]\label{lem:three-faces}
Let $F$ be a fullerene graph.
\begin{enumerate}[label=\textup{(\alph*)}]
  \item\label{it:faces-a} Two distinct faces of $F$ share at most one edge; if they share an edge, they share no further vertex; and if they share a vertex, they share an edge incident with it. In particular, the vertex sets of two distinct faces intersect in $\emptyset$, a single vertex, or the two endpoints of a shared edge.
  \item\label{it:faces-b} If three distinct faces of $F$ pairwise share a vertex, then all three share a common vertex.
\end{enumerate}
\end{observation}

Both parts are standard facts about the face structure of fullerene graphs: (a) is the usual intersection pattern of faces in a $3$-connected cubic plane graph---in particular the three faces at a vertex pairwise share one of the three edges at that vertex, so faces meeting in a vertex meet in an edge incident with it---and (b) expresses that in the dual triangulation three pairwise-adjacent vertices form a facial rather than a separating triangle, the latter being excluded by cyclic $5$-edge-connectivity (cf.\ \cite{andovaMathematicalAspectsFullerenes2016,cyclic_edge_cuts}).


The next technical lemma is a formalisation of the intuition that large nontrivial partial automorphisms with small support cannot be contained within a single face of an IPR fullerene, because the unique exterior neighbours of the vertices in the support would have to be fixed by the partial automorphism, which would force the vertices in the support to be fixed as well.

\begin{lemma}[Support is not face-localised]\label{lem:support-not-face-localised}
Let $F$ be an IPR fullerene on $n$ vertices, and let $\varphi$ be a nontrivial partial automorphism of $F$ with $\rank(\varphi) = n - k$ and $k \le 3$.  Then $\varphi$ is not face-localised: $\supp(\varphi) \not\subseteq V(F^\circ)$ for any face $F^\circ$ of $F$.
\end{lemma}

\begin{proof}
Suppose for contradiction that $\supp(\varphi) \subseteq V(F^\circ)$ for some face $F^\circ$ of $F$, which is either a pentagon ($|V(F^\circ)| = 5$) or a hexagon ($|V(F^\circ)| = 6$).  Let $D := \dom(\varphi)$ and $S_D := V(F) \setminus D$, so $|S_D| = k \le 3$.

Since fullerenes have girth at least $5$, the face boundary $C_{F^\circ} = \partial F^\circ$ is a chordless cycle in $F$.  Hence $F[V(F^\circ)] = C_{F^\circ}$, and $\varphi$ restricted to $D \cap V(F^\circ)$ is a graph isomorphism between two induced subgraphs of the cycle $C_{F^\circ}$.  Vertices of $V(F) \setminus V(F^\circ)$ that lie in $D$ are fixed by $\varphi$ (since $\supp(\varphi) \subseteq V(F^\circ)$).

For every $v \in V(F^\circ) \cap D$, the unique exterior neighbour $\ex(v)$ is either in $S_D$ (i.e., removed from the domain) or fixed by $\varphi$ (it lies outside $V(F^\circ) \supseteq \supp(\varphi)$).  Moreover, if $v$ is moved, then $\varphi(v) \in \supp(\varphi) \subseteq V(F^\circ)$ by the symmetric convention on supports.  If $\ex(v)$ is fixed, the edge $v \, \ex(v)$ in $F$ must map to the edge $\varphi(v) \, \varphi(\ex(v)) = \varphi(v) \, \ex(v)$, so $\ex(v)$ is also a neighbour of $\varphi(v)$; since $\ex(v) \notin V(F^\circ)$ and $\varphi(v) \in V(F^\circ)$, it is the exterior neighbour of $\varphi(v)$.  This gives the key identity
\begin{equation}\label{eq:primal-exterior}
  \ex(v) = \ex(\varphi(v))
  \qquad\text{whenever } \ex(v) \in D.
\end{equation}
By \Cref{lem:distinct-exterior} the map $\ex$ is injective; combined with~\eqref{eq:primal-exterior} this forces $\varphi(v) = v$ whenever $\ex(v) \in D$.

Let $M := \{ v \in V(F^\circ) \cap D \mid \varphi(v) \ne v \}$ denote the set of vertices of $F^\circ$ actually moved by $\varphi$.  By the previous paragraph, $\ex(v) \in S_D$ for every $v \in M$, and by \Cref{lem:distinct-exterior} the exterior vertices $\{\ex(v) : v \in M\}$ are pairwise distinct.  Hence
\begin{equation}\label{eq:primal-S_D-lower}
  |S_D \cap (V(F) \setminus V(F^\circ))| \;\ge\; |M|.
\end{equation}

Because $\varphi$ is nontrivial and $\supp(\varphi) \subseteq V(F^\circ)$, the restriction $\sigma := \varphi|_{V(F^\circ) \cap D}$ is a non-identity graph isomorphism between two induced subgraphs of the cycle $C_{F^\circ}$.  Write $X := V(F^\circ) \cap S_D$ and $A := V(F^\circ) \cap D = V(F^\circ) \setminus X$, so $|X| \le |S_D| = k \le 3$.

\medskip
\noindent\textit{Claim.} $|X| + |M| \ge 4$.

\smallskip
Since $\sigma$ is non-identity, $|M| \ge 1$, so the claim holds when $|X| = 3$. The remaining values $|X| \le 2$ are settled by the two cases below: the case $|X| \le 1$ is treated directly, and for $|X| = 2$ it suffices to show $|M| \ge 2$, which holds because the case $|M| = 1$ below forces $|X| \ge 3$, incompatible with $|X| = 2$. Throughout, $\sigma(A) \subseteq V(F^\circ)$: a moved vertex maps into $\supp(\varphi) \subseteq V(F^\circ)$, and a fixed vertex is its own image.

\smallskip
\noindent\textit{Case $|X| \le 1$.}\ Identify $V(F^\circ)$ with $\mathbb{Z}_L$, $L := |V(F^\circ)| \in \{5, 6\}$, along the cycle. Here $A$ is the whole cycle or an induced path on $L - 1$ vertices, and $\sigma$ maps consecutive vertices of $A$ to adjacent vertices of $C_{F^\circ}$, so $\sigma(p) = \varepsilon p + c \pmod{L}$ on $A$ for some $\varepsilon \in \{\pm 1\}$, $c \in \mathbb{Z}_L$. If $\varepsilon = +1$, a fixed point forces $c = 0$ and $\sigma = \mathrm{id}$, which is excluded; hence $|M| = |A| = L - |X|$ and $|X| + |M| = L \ge 5$. If $\varepsilon = -1$, the fixed points satisfy $2p \equiv c \pmod{L}$, which has at most one solution for $L = 5$ and at most two for $L = 6$; hence $|X| + |M| \ge L - 1 = 4$ for $L = 5$ and $|X| + |M| \ge L - 2 = 4$ for $L = 6$.

\smallskip
\noindent\textit{Case $|M| = 1$.}\ Write $M = \{v\}$ and $w' := \varphi(v) \ne v$; then $w' \in \supp(\varphi) \subseteq V(F^\circ)$, and $w' \notin D$---for if $w' \in D$, then $\varphi(w') \neq w'$ would place $w'$ in $M$, contradicting $M = \{v\}$ (as $w' \neq v$), while $\varphi(w') = w'$ would give two distinct domain vertices $v \neq w'$ with the same image $w'$, contradicting injectivity of $\varphi$---so $w' \in X$; moreover $w'$ is the only moved-onto vertex. By \Cref{lem:anchor} (applied to $\varphi$ at $v$ and to $\varphi^{-1}$ at $w'$), every fixed neighbour of $v$ or of $w'$ is a common neighbour of $v$ and $w'$. But $v$ and $w'$ have at most one common neighbour, and only the middle vertex of a $v$--$w'$ path of length two along $C_{F^\circ}$ can be one: adjacent vertices have none (girth), a second common neighbour would close a $4$-cycle, and an off-cycle common neighbour would force $\ex(v) = \ex(w')$, contradicting \Cref{lem:distinct-exterior}. Hence every cycle-neighbour of $v$ or of $w'$, other than this possible middle vertex and other than $v, w'$ themselves, is neither fixed nor moved, and so lies in $X$. In each of the three relative positions of $v$ and $w'$ on the cycle (adjacent, at cycle-distance $2$, at cycle-distance $3$), these vertices together with $w'$ give at least three distinct elements of $X$; hence $|X| \ge 3$, and therefore $|X| + |M| = |X| + 1 \ge 4$.

\medskip
This proves the claim.  Combining the claim with~\eqref{eq:primal-S_D-lower} and the fact that the exterior vertices $\{\ex(v) : v \in M\}$ are $|M|$ distinct elements of $S_D \setminus V(F^\circ)$:
\[
  k \;=\; |S_D|
    \;=\; |X| + |S_D \setminus V(F^\circ)|
    \;\ge\; |X| + |M|
    \;\ge\; 4.
\]
This contradicts $k \le 3$. Therefore $\supp(\varphi) \not\subseteq V(F^\circ)$ for any face $F^\circ$ of $F$.
\end{proof}

\subsection{Partial automorphisms of small deficiency are not localised}\label{subsec:unified}

We now prove the main structural result of this section. It says that a partial automorphism of an IPR fullerene of small deficiency can never confine its nontrivial action to a region attached to the pointwise-fixed remainder of the cage by a small edge interface, nor to a single face. The statement requires neither the asymmetry of $F$ nor the extremality of $\varphi$, and covers every partial automorphism of small deficiency---including those whose domain and range differ. The structure theorem for asymmetric IPR fullerenes of asymmetric depth $2$ and $3$ announced in the introduction is the corollary recorded in \Cref{cor:hidden-symmetry-primal} below.

Throughout, $F$ is an IPR fullerene on $n$ vertices and $\varphi$ is a nontrivial partial automorphism of $F$ with domain $D = \dom(\varphi)$ and $S_D = V(F) \setminus D$. We call
\[
  k \;:=\; n - \rank(\varphi) \;=\; |S_D|
\]
the \emph{deficiency} of $\varphi$; if $F$ is asymmetric and $\varphi$ realises $d(F)$, then $k = d(F)$. Recall that a moved vertex cannot have a fixed image, so $\varphi(\supp(\varphi)) \subseteq \supp(\varphi) \cup S_D$; the vertices of $\ran(\varphi) \setminus D$ are called \emph{exit vertices} and those of $D \setminus \ran(\varphi)$ \emph{entry vertices}. Exit vertices are the reason why a partial automorphism need not decompose as $\alpha \cup \id_R$: a moved vertex may be mapped onto a deleted vertex.

\begin{theorem}[Non-localisation for small deficiency]\label{thm:unified-localisation}
Let $F$ be an IPR fullerene and let $\varphi$ be a nontrivial partial automorphism of $F$ of deficiency $k \in \{1,2,3\}$. Then:
\begin{enumerate}[label=\textup{(\roman*)}]
  \item\label{it:unified-face} $\varphi$ is not face-localised: $\supp(\varphi) \not\subseteq V(F^\circ)$ for every face $F^\circ$ of $F$;
  \item\label{it:unified-interface} there is no partition $D = C \sqcup R$ such that $\supp(\varphi) \subseteq C$, the set $R$ is fixed pointwise, $|R| \ge 6$, and at most $5-k$ edges of $F$ join $C$ to $R$.
\end{enumerate}
\end{theorem}

Clause~\ref{it:unified-interface} assumes no invariance of $C$: the vertices of $C$ may be moved anywhere, including into $S_D$. Informally, deleting at most three vertices of an IPR fullerene never buys a symmetry that acts only behind a small edge interface. For $k = 0$ the corresponding statement is classical rigidity, since an automorphism of a $3$-connected plane graph fixing a face pointwise is the identity; \Cref{thm:unified-localisation} says that localisation does not become possible when up to three vertices are sacrificed. The proof is partly computer-assisted: clause~\ref{it:unified-face} and the cyclic cases of clause~\ref{it:unified-interface} are proved by hand, while the acyclic (forest) case of clause~\ref{it:unified-interface} is reduced to a finite enumeration that is discharged by the exhaustive verification of \Cref{lem:acyclic-case}; the enumeration, its independent checks, and the archived code are described in \S\ref{sec:computational-methodology}.

\begin{remark}[The two clauses are independent]\label{rem:unified-independence}
Neither clause implies the other. A face-localised $\varphi$ need not decompose as in~\ref{it:unified-interface} at all: if $\varphi$ rotates a hexagonal face $F^\circ$ by two positions then, by \Cref{lem:anchor}, no moved vertex has a fixed neighbour, so the smallest set containing its support is $V(F^\circ)$, whose interface to the fixed remainder already has $6-k > 5-k$ edges; conversely~\ref{it:unified-interface} excludes configurations whose support lies nowhere near a single face.
\end{remark}

\begin{remark}[The scale of clause~\ref{it:unified-interface}]\label{rem:unified-scale}
It is worth recording explicitly how large the excluded regions may be. The deleted set contributes at most $3k$ edges to the two boundaries, so~\eqref{eq:unified-budget} caps $|\partial C| + |\partial R|$ at $10 + k \le 13$; since every nontrivial cyclic edge cut of an IPR fullerene has at least seven edges, $\partial C$ and $\partial R$ cannot both be nontrivial cyclic cuts. This is what forces one of $C$, $R$ to be a face boundary or its complement when both sides induce cycles, and what confines the remaining case to $|C| \le 3+2k \le 9$ by~\eqref{eq:size-bound}. In words: the support of $\varphi$ can hide neither inside a single face nor inside a patch of at most nine vertices attached to the rest of the cage by at most $5-k$ edges. Whether the interface bound $5-k$ can be raised---ideally towards the cyclic edge-connectivity threshold, which would make the conclusion global in a literal sense---we do not know.
\end{remark}

\subsubsection*{Preliminaries for the proof}

Before proving \Cref{thm:unified-localisation} we record one lemma and collect the standard fullerene facts the argument uses. The lemma describes how rigidly a partial automorphism behaves at a vertex mapped outside its own domain; it is what lets us treat partial automorphisms with $\ran(\varphi) \neq \dom(\varphi)$, and so drop the hypothesis $\varphi(C) = C$.

\begin{observation}[Exits]\label{lem:exit}
Let $\varphi$ be a partial automorphism of a graph $F$, let $x \in \supp(\varphi)$ and let $e := \varphi(x)$ be an exit vertex. Then
\[
  N(e) \cap \ran(\varphi) \;=\; \varphi\bigl(N(x) \cap \dom(\varphi)\bigr).
\]
In particular, since $\fix(\varphi) \subseteq \ran(\varphi)$, the adjacencies of $e$ to all fixed vertices, to all images of moved vertices and to all other exit vertices are determined by those of $x$; the undetermined edges of $e$ can only join $e$ to entry vertices or to vertices of $S_D \setminus \ran(\varphi)$.
\end{observation}

\begin{proof}
Every vertex of $\ran(\varphi)$ is of the form $\varphi(z)$ with $z \in \dom(\varphi)$, and $\varphi$ is an isomorphism of the induced subgraphs $F[\dom(\varphi)] \to F[\ran(\varphi)]$; hence $e \sim \varphi(z)$ if and only if $x \sim z$.
\end{proof}

We also use three standard facts about the short cycles of a fullerene; only the last of them uses the IPR hypothesis. First, \emph{every $5$-cycle and every $6$-cycle of a fullerene bounds a face}. Indeed, let $Z$ be a cycle of length $L \in \{5,6\}$. Girth $5$ makes $Z$ chordless, since a chord would split $Z$ into two cycles of lengths summing to $L+2 \le 8$, one of them of length at most $4$. Of the $L$ edges leaving $Z$, say $t$ go to one side of $Z$ and $L-t$ to the other; if $t = 0$ or $t = L$ then, $Z$ being chordless, the corresponding side contains no vertex and $Z$ bounds a face. Otherwise let $X$ be the vertex set of the side receiving $m := \min(t, L-t) \le \lfloor L/2 \rfloor \le 3$ edges, so that $X \ne \emptyset$ and $|\partial X| = m \le 3$. If $F[X]$ contains a cycle then $\partial X$ is a cyclic edge cut of size at most $3$, contradicting cyclic $5$-edge-connectivity; and if $F[X]$ is a forest with $t'$ components then $|\partial X| = 3|X| - 2(|X|-t') = |X| + 2t' \ge 3$, forcing $m = 3$, hence $L = 6$, $|X| = 1$ and $t' = 1$, so that $X$ is a single vertex adjacent to three vertices of the $6$-cycle $Z$; but among any three vertices of a $6$-cycle two are at cycle-distance at most $2$, and together with that vertex they close a triangle or a $4$-cycle, again contradicting girth $5$. Second, since in a cubic plane graph the three faces incident with a vertex pairwise share an edge, two faces meeting in a vertex meet in an edge. Third, distinct pentagons being non-adjacent in an IPR fullerene, and distinct $5$-cycles bounding distinct pentagonal faces by the first fact, distinct $5$-cycles of $F$ are vertex-disjoint.

\subsubsection*{Proof of \Cref{thm:unified-localisation}}

\emph{Clause~\ref{it:unified-face}} is \Cref{lem:support-not-face-localised} applied with the deficiency $k \le 3$. (That lemma uses only the girth of $F$, \Cref{lem:distinct-exterior} and $|S_D| \le 3$.)

\emph{Clause~\ref{it:unified-interface}.} Suppose such a partition $D = C \sqcup R$ exists. Write $c \le 5-k$ for the number of $C$--$R$ edges and let $e_C$ and $e_R$ be the numbers of edges joining $S_D$ to $C$ and to $R$, so that $e_C + e_R \le 3k$. Since $V(F) = C \sqcup R \sqcup S_D$, the edge boundaries satisfy $|\partial C| = c + e_C$ and $|\partial R| = c + e_R$, whence
\begin{equation}\label{eq:unified-budget}
  |\partial C| + |\partial R| \;=\; 2c + e_C + e_R \;\le\; 2(5-k) + 3k \;=\; 10 + k \;\le\; 13 \;<\; 7 + 7 .
\end{equation}
We distinguish three cases according to whether $F[C]$ and $F[R]$ contain cycles.

\smallskip
\emph{Case 1: both $F[C]$ and $F[R]$ contain a cycle} (\Cref{fig:edge-cut-budget}). Then $\partial C$ and $\partial R$ are cyclic edge cuts, and by~\eqref{eq:unified-budget} at least one of them has size at most~$6$; being a cyclic edge cut of a fullerene it has size at least~$5$, hence size in $\{5,6\}$, and by~\cite{cyclic_edge_cuts} it is trivial. Deleting it leaves exactly two components, the vertex set $V(P)$ of a face $P$ with $|V(P)| \in \{5,6\}$ and its complement, and the corresponding side ($C$ or $R$) equals one of them. We rule out the four possibilities.
\begin{itemize}
  \item $C = V(P)$: then $\supp(\varphi) \subseteq C = V(P)$ is confined to a single face, contradicting~\ref{it:unified-face}.
  \item $C = V(F) \setminus V(P)$: then $R \cup S_D = V(P)$. As $F[V(P)]$ is a chordless cycle, its only subset inducing a cycle is $V(P)$ itself, so $R = V(P)$ and $S_D = \emptyset$, contradicting $k \ge 1$.
  \item $R = V(P)$: then $|R| \le 6$, so $|R| = 6$ and $P$ is a hexagon. By \Cref{lem:distinct-exterior} the six exterior edges of $P$ end at six pairwise distinct vertices of $C \cup S_D$. Those ending in $C$ are among the $c$ edges joining $C$ to $R$, and those ending in $S_D$ number at most $k$; hence $6 \le c + k \le (5-k) + k = 5$, a contradiction.
  \item $R = V(F) \setminus V(P)$: then $C \cup S_D = V(P)$, so $\supp(\varphi) \subseteq C \subseteq V(P)$, again contradicting~\ref{it:unified-face}.
\end{itemize}

\begin{figure}[!ht]
\centering
\begin{tikzpicture}[scale=0.95, lbl/.style={font=\small}, sd/.style={rectangle, draw=black, fill=gray!30, inner sep=2.2pt}]
  \draw (-2.2,0) ellipse (1.5 and 1.0);
  \draw (2.2,0) ellipse (1.5 and 1.0);
  \node[lbl] at (-2.2,-0.05) {$C$: moved};
  \node[lbl] at (2.2,0.05) {$R$: fixed pointwise};
  \draw (-0.75,0.35)--(0.75,0.35);
  \draw (-0.75,-0.35)--(0.75,-0.35);
  \node[lbl] at (0,0) {$c \le 5-k$};
  \coordinate (s1) at (-0.9,2.0); \coordinate (s2) at (0,2.0); \coordinate (s3) at (0.9,2.0);
  \foreach \p in {s1,s2,s3} \node[sd] at (\p) {};
  \node[lbl, above] at (0,2.15) {$S_D$, $|S_D| = k \le 3$};
  \draw[dashed] (s1)--(-1.9,0.95);
  \draw[dashed] (s2)--(-1.4,0.95);
  \draw[dashed] (s2)--(1.4,0.95);
  \draw[dashed] (s3)--(1.9,0.95);
  \node[lbl] at (-1.9,1.6) {$e_C$};
  \node[lbl] at (1.9,1.6) {$e_R$};
  \node[lbl] at (0,-1.5) {$|\partial C| + |\partial R| = 2c + e_C + e_R \le 2(5-k) + 3k = 10+k \le 13 < 7 + 7$};
\end{tikzpicture}
\caption{The counting~\eqref{eq:unified-budget} in Case~1 of the proof of \Cref{thm:unified-localisation}\ref{it:unified-interface}: the two edge cuts $\partial C$ and $\partial R$ cannot both be cyclic edge cuts of size at least~$7$, so the smaller has size in $\{5,6\}$ and is trivial (\cite{cyclic_edge_cuts}); then one of $C$, $R$ is a face boundary or its complement, and each of the four possibilities is ruled out. The dashed edges are those joining the deleted set $S_D$ to $C$ and to $R$, at most $3k$ in all.}
\label{fig:edge-cut-budget}
\end{figure}

\smallskip
\emph{Case 2: $F[C]$ contains a cycle and $F[R]$ is a forest.} If $X \subseteq V(F)$ induces a forest in the cubic graph $F$, then $|\partial X| = 3|X| - 2\,|E(F[X])| \ge 3|X| - 2(|X|-1) = |X| + 2$. Thus $|\partial R| \ge |R| + 2 \ge 8$, and~\eqref{eq:unified-budget} gives $|\partial C| \le (10+k) - 8 = k + 2 \le 5$. If $R \cup S_D$ induced a forest, the same bound would give $|\partial(R \cup S_D)| \ge |R| + k + 2 \ge 9$, while $\partial(R \cup S_D) = \partial C$ has at most five edges; hence $F[R \cup S_D]$ contains a cycle, both sides of $\partial C$ contain cycles, and $\partial C$ is a cyclic edge cut of size exactly~$5$, therefore trivial. Either $C = V(P)$, contradicting~\ref{it:unified-face}, or $R \cup S_D = V(P)$, which is impossible because $|R \cup S_D| \ge 6 + k \ge 7 > 6 \ge |V(P)|$.

\smallskip
\emph{Case 3: $F[C]$ is a forest.} Let $T := F[C]$ with $t$ components. Counting the degrees of the vertices of $C$ in the cubic graph $F$,
\begin{equation}\label{eq:size-bound}
  3|C| \;=\; 2\bigl(|C| - t\bigr) + c + e_C,
  \qquad\text{hence}\qquad
  |C| + 2t \;=\; c + e_C \;\le\; (5-k) + 3k \;=\; 5 + 2k ,
\end{equation}
so that $|C| \le 3 + 2k \le 9$. The set $C$ is therefore bounded in size, and the case reduces to a finite verification, carried out in \Cref{lem:acyclic-case} below. This completes the proof. \qed

\subsubsection*{The finite case: admissible configurations}

The following notation records all the information that a configuration as in Case~3 imposes on $F$. Write $T := F[C]$ for the induced forest, $M := \supp(\varphi) \subseteq C$ for the moved vertices, $W := C \setminus M$ for the vertices of $C$ fixed by $\varphi$, and $\psi := \varphi|_M$; denote the deleted vertices by $s_1, \dots, s_k$. The finite object
\[
  \Gamma \;=\; \bigl(T,\; M,\; \psi,\; (N_C(\rho))_{\rho \in R},\; (N(v) \cap S_D)_{v \in C}\bigr)
\]
is called a \emph{configuration record}. Its \emph{recorded subgraph} $K$ has vertex set $C \cup \{\rho \in R : N_C(\rho) \neq \emptyset\} \cup \{s_1, \dots, s_k\}$ and contains every edge listed in $\Gamma$ together with every edge forced by~\ref{A3} and~\ref{A4} below. Each edge of $K$ is an edge of $F$, but $K$ need not contain all edges of $F$ between its vertices; accordingly every condition below refers only to the \emph{recorded} edges, so that an edge of $F$ absent from $K$ can never invalidate a rejection.

The next list is what the finite verification checks. Conditions \ref{A1}, \ref{A2} and \ref{A5} are requirements the record must satisfy; \ref{A3} and \ref{A4} instead \emph{force} further edges into $K$ (a record is discarded if a forced edge is impossible). Each item is an immediate consequence of the definition of a partial automorphism together with one of the facts established above, as indicated on the right. Although $F[C]$ is a forest, the recorded subgraph $K$ is typically not acyclic: the deleted vertices $s_j$ and the interface vertices $\rho$ close up cycles running through $C$ (for instance, two vertices of $C$ with a common neighbour in $S_D$, together with the $T$-path between them). Condition~\ref{A5} is the only place where the pentagon/hexagon face structure of $F$ enters, and it is exactly what those cycles must respect; it is what makes the enumeration terminate in a contradiction, the girth and degree conditions~\ref{A1}--\ref{A4} alone being insufficient.

\begin{enumerate}[label=\textup{(A\arabic*)},leftmargin=*]
  \item\label{A1} \emph{Degrees and budgets.} $T$ is a forest with $\Delta(T) \le 3$; every $v \in C$ satisfies $\deg_T(v) + |N(v) \cap R| + |N(v) \cap S_D| = 3$; moreover $\sum_{\rho \in R} |N_C(\rho)| = c \le 5-k$, $|C| + 2t \le 5+2k$, and each $s_j$ is incident with at most three edges of $K$. \hfill($F$ cubic, $V(F) = C \sqcup R \sqcup S_D$;~\eqref{eq:size-bound})
  \item\label{A2} \emph{The map.} $\psi : M \to M \cup S_D$ is injective and fixed-point-free, and sends no vertex of $M$ into $W$. For $x, y \in M$ with $\psi(x), \psi(y) \in M$ we have $x \sim y \iff \psi(x) \sim \psi(y)$, and $N_W(x) = N_W(\psi(x))$. \hfill(injectivity; $\varphi$ preserves adjacency on $D$; \Cref{lem:anchor})
  \item\label{A3} \emph{Forced edges at exits and at the interface.} If $\psi(x) = e \in S_D$, then for every $y \in C \setminus \{x\}$ the pair $\{e, \varphi(y)\}$ is an edge of $F$ exactly when $x \sim y$, where $\varphi(y)$---the image of $y$---equals $\psi(y)$ if $y$ is moved and $y$ itself if $y$ is fixed; and every $\rho \in R$ with $N_C(\rho) \neq \emptyset$ satisfies $\rho \sim x \iff \rho \sim \psi(x)$ for all $x \in M$. Edges so determined are added to $K$ and the corresponding non-edges are recorded as forbidden. \hfill(\Cref{lem:exit}; $\rho$ is fixed)
  \item\label{A4} \emph{Exit completion.} Let $e = \psi(x) \in S_D$ be an exit vertex. By \Cref{lem:exit} the neighbours of $e$ in $\ran(\varphi)$ are exactly the images of the neighbours of $x$, all already recorded, so any \emph{undetermined} neighbour of $e$ lies outside $\ran(\varphi)$, that is, is an entry vertex or a vertex of $S_D \setminus \ran(\varphi)$. Entry vertices lie in $C$, and every vertex of $C$ already has all three of its edges recorded by~\ref{A1}; hence no entry vertex can be an undetermined neighbour of $e$. Thus all $3 - \deg_K(e)$ undetermined edges of $e$ join it to $S_D \setminus \ran(\varphi)$: the record is discarded unless $3 - \deg_K(e) \le |S_D \setminus \ran(\varphi)|$, and when equality holds each vertex of $S_D \setminus \ran(\varphi)$ must receive one of them, so these edges are forced and added to $K$. The completion is iterated until no further edge is forced. \hfill(\Cref{lem:exit})
  \item\label{A5} \emph{Girth and faces.} $K$ contains no cycle of length $3$ or $4$; distinct $5$-cycles of $K$ are vertex-disjoint; every $6$-cycle of $K$ all of whose chords are certifiably absent---a chord being certifiably absent when one of its endpoints has recorded degree~$3$---bounds a hexagonal face; two recognised faces meet in $\emptyset$ or in exactly one edge with its endpoints; no vertex lies on more than three and no edge on more than two recognised faces. \hfill(girth~$5$; the standard face facts above; \Cref{lem:three-faces})
\end{enumerate}

\begin{lemma}[Computational]\label{lem:acyclic-case}
For $k \in \{1,2,3\}$ and $c \le 5-k$ no configuration record satisfies \ref{A1}--\ref{A5}. Consequently Case~3 of the proof of \Cref{thm:unified-localisation}\ref{it:unified-interface} cannot occur.
\end{lemma}

\begin{proof}[Proof (exhaustive verification)]
By~\eqref{eq:size-bound} it suffices to enumerate forests $T$ on at most $3+2k \le 9$ vertices with $\Delta(T) \le 3$ and $|C| + 2t \le 5+2k$ (obtained from the subcubic trees generated by \texttt{gentreeg}), and, for each, all sets $M$, all maps $\psi$, all interface patterns and all $S_D$-attachments compatible with \ref{A1}--\ref{A3}; the forced-edge closure~\ref{A4} is then applied and the conditions~\ref{A5} evaluated. All conditions are invariant under relabelling of $C$, of $R$ and of $S_D$, so enumeration up to isomorphism is sufficient. The run for deficiency $k$ subsumes all smaller deficiencies, since unused deleted vertices are simply never referenced.

The enumeration examines $6\,623$ pairs $(T, M)$, $917\,415$ maps $\psi$ and $12\,005\,412$ attachment patterns, yielding $12\,183$ complete configuration records. Of these, $1\,266$ are eliminated by the exit completion~\ref{A4} and $2\,829$ by the girth condition in~\ref{A5}; each of the remaining $8\,088$ records is eliminated by the face conditions in~\ref{A5}. No configuration record survives. Details of the implementation and of the independent checks performed on it are given in \S\ref{sec:computational-methodology}.
\end{proof}

\subsubsection*{Sharpness: what the hypotheses are really doing}

The configuration records that survive every condition except the face conditions of~\ref{A5} are locally consistent subgraphs of cubic girth-$5$ graphs. The following proposition shows that they are not artefacts of the method: one of them extends to a genuine graph, so \Cref{thm:unified-localisation} is false for cubic, $3$-connected graphs of girth~$5$ in general.

\begin{proposition}\label{prop:girth5-witness}
There exists an asymmetric cubic $3$-connected graph $G$ of girth~$5$ on $20$ vertices admitting a nontrivial partial automorphism $\varphi$ of deficiency~$3$ such that $\supp(\varphi)$ has four vertices and $\dom(\varphi)$ splits into a set $C$ with $|C| = 5$ containing $\supp(\varphi)$ and a pointwise-fixed remainder $R$ with $|R| = 12$ joined to $C$ by exactly two edges. In particular the conclusion of \Cref{thm:unified-localisation}\ref{it:unified-interface} fails for $G$.
\end{proposition}

\begin{proof}
Take the configuration of \Cref{fig:girth5-witness}: a path $x\,u\,w\,v\,y$ carrying the reversal $\psi = (x\,y)(u\,v)$, a fixed vertex $\rho$ adjacent to $x$ and to $y$, and deleted vertices $s_1 \sim u,y$, $s_2 \sim v,x$ and $s_3 \sim w$. The map $\varphi$ that acts as $\psi$ on $\{x,u,w,v,y\}$ and fixes every vertex outside $\{s_1,s_2,s_3\}$ is a partial automorphism of any cubic graph containing this configuration: adjacency and non-adjacency inside the path are preserved by the reversal, and the only edges from the path to the fixed remainder are $\rho x$ and $\rho y$, which the reversal interchanges. An explicit completion to an asymmetric cubic $3$-connected graph of girth~$5$ on $20$ vertices is given in \S\ref{sec:computational-methodology}. Since the graph is asymmetric, $\varphi$ is not the restriction of any automorphism.
\end{proof}

The configuration of \Cref{prop:girth5-witness} cannot occur in any fullerene, IPR or not. In it the two $5$-cycles $x\,u\,w\,v\,s_2$ and $u\,w\,v\,y\,s_1$ share the two edges $uw$ and $wv$. If both bounded faces, these would be two faces meeting in two edges, which is impossible in fullerenes; so at least one of them bounds no face, that is, is a \emph{separating pentagon}. But in a fullerene every $5$-cycle bounds a face, so no separating pentagon exists. Hence \emph{this particular} configuration is excluded in every fullerene, IPR or not, and the graph of \Cref{prop:girth5-witness} has no fullerene counterpart. This does not by itself settle clause~\ref{it:unified-interface} for non-IPR fullerenes; see \Cref{q:sharpness}.

Whether \Cref{thm:unified-localisation} extends to non-IPR fullerenes we do not know. As recorded in the preliminaries to the proof, the facts that every $5$-cycle and every $6$-cycle bounds a face hold in \emph{every} fullerene, so the IPR hypothesis enters in exactly one place: Cases~1 and~2 of clause~\ref{it:unified-interface} use that every cyclic $5$- or $6$-edge cut of $F$ is trivial, which by~\cite{cyclic_edge_cuts} fails precisely when two pentagons are adjacent. Removing the hypothesis therefore reduces to analysing the nontrivial cyclic $5$- and $6$-edge cuts described by the Kardo\v{s}--\v{S}krekovski characterisation.

\begin{question}\label{q:sharpness}
Does the conclusion of \Cref{thm:unified-localisation} hold for all fullerenes, not only the IPR ones? More generally, does a localised partial automorphism as in \ref{it:unified-interface} exist in any $3$-connected planar cubic graph of girth~$5$?
\end{question}

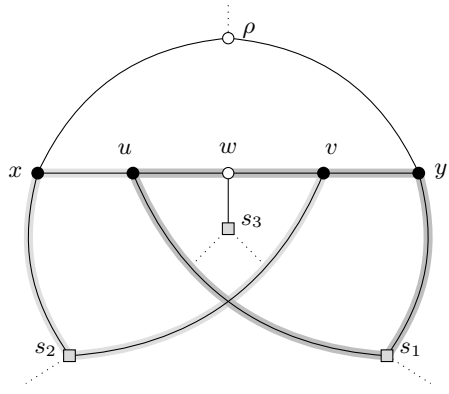
\begin{figure}[!ht]
\centering
\begin{tikzpicture}[scale=1.05,
  mv/.style={circle, fill=black, inner sep=1.7pt},
  fx/.style={circle, draw=black, fill=white, inner sep=1.5pt},
  sd/.style={rectangle, draw=black, fill=gray!30, inner sep=2.2pt},
  lbl/.style={font=\small}]
  \coordinate (x) at (-2.4,0);
  \coordinate (u) at (-1.2,0);
  \coordinate (w) at (0,0);
  \coordinate (v) at (1.2,0);
  \coordinate (y) at (2.4,0);
  \coordinate (rho) at (0,1.7);
  \coordinate (s3) at (0,-0.7);
  \coordinate (s2) at (-2.0,-2.3);
  \coordinate (s1) at (2.0,-2.3);
  \draw[line width=3.4pt, gray!25] (x)--(u)--(w)--(v);
  \draw[line width=3.4pt, gray!25] (x) to[bend right=25] (s2);
  \draw[line width=3.4pt, gray!25] (s2) to[bend right=32] (v);
  \draw[line width=3.4pt, gray!50] (u)--(w)--(v)--(y);
  \draw[line width=3.4pt, gray!50] (y) to[bend left=25] (s1);
  \draw[line width=3.4pt, gray!50] (s1) to[bend left=32] (u);
  \draw (x)--(u)--(w)--(v)--(y);
  \draw (x) to[bend left=30] (rho);
  \draw (y) to[bend right=30] (rho);
  \draw (w)--(s3);
  \draw (x) to[bend right=25] (s2);
  \draw (s2) to[bend right=32] (v);
  \draw (y) to[bend left=25] (s1);
  \draw (s1) to[bend left=32] (u);
  \draw[dotted] (rho)--++(0,0.55);
  \draw[dotted] (s3)--++(-0.45,-0.45);
  \draw[dotted] (s3)--++(0.45,-0.45);
  \draw[dotted] (s2)--++(-0.55,-0.35);
  \draw[dotted] (s1)--++(0.55,-0.35);
  \node[mv] at (x) {}; \node[mv] at (u) {};
  \node[fx] at (w) {}; \node[mv] at (v) {}; \node[mv] at (y) {};
  \node[fx] at (rho) {};
  \node[sd] at (s1) {}; \node[sd] at (s2) {}; \node[sd] at (s3) {};
  \node[lbl] at ($(x)+(-0.28,0.02)$) {$x$};
  \node[lbl] at ($(u)+(-0.10,0.30)$) {$u$};
  \node[lbl] at ($(w)+(0.00,0.32)$) {$w$};
  \node[lbl] at ($(v)+(0.10,0.30)$) {$v$};
  \node[lbl] at ($(y)+(0.28,0.02)$) {$y$};
  \node[lbl] at ($(rho)+(0.26,0.10)$) {$\rho$};
  \node[lbl] at ($(s3)+(0.30,0.10)$) {$s_3$};
  \node[lbl] at ($(s2)+(-0.30,0.08)$) {$s_2$};
  \node[lbl] at ($(s1)+(0.30,0.08)$) {$s_1$};
\end{tikzpicture}
\caption{The configuration of \Cref{prop:girth5-witness}: a nontrivial partial automorphism of deficiency~$3$ localised behind a two-edge interface, possible in a cubic $3$-connected graph of girth~$5$ but not in any fullerene. Black vertices are moved by the reversal $\psi = (x\,y)(u\,v)$, white vertices are fixed, squares are the deleted vertices $S_D$; dotted stubs lead to the pointwise-fixed remainder $R$, which meets $C = \{x,u,w,v,y\}$ only in the two edges $\rho x$ and $\rho y$. The two shaded $5$-cycles $x\,u\,w\,v\,s_2$ and $u\,w\,v\,y\,s_1$ share the two edges $uw$ and $wv$; in a fullerene both would bound pentagonal faces, which is impossible. Crossings in the drawing are artefacts of the layout.}
\label{fig:girth5-witness}
\end{figure}

\subsubsection*{Consequences}

\begin{corollary}[Decomposable form]\label{cor:unified-decomposable}
Let $F$ and $\varphi$ be as in \Cref{thm:unified-localisation}. Then $\varphi$ does not decompose as $\varphi = \alpha \cup \id_R$, where $\alpha$ is a nontrivial automorphism of an induced subgraph $F[C]$ with $C \subset \dom(\varphi)$, the remainder $R = \dom(\varphi) \setminus C$ is fixed pointwise with $|R| \ge 6$, and at most $5-k$ edges of $F$ join $C$ to $R$.\end{corollary}

\begin{proof}
Such a decomposition satisfies $\supp(\varphi) \subseteq C$ and $\varphi(C) = C$, hence is a special case of the partition excluded by \Cref{thm:unified-localisation}\ref{it:unified-interface}.
\end{proof}

\begin{corollary}[Non-localisation of depth-realising partial automorphisms]\label{cor:hidden-symmetry-primal}
Let $F$ be an asymmetric IPR fullerene on $n$ vertices with $d(F) \in \{2,3\}$, and let $\varphi$ be a partial automorphism of $F$ realising $d(F)$, that is, $\rank(\varphi) = n - d(F)$. Then:
\begin{enumerate}[label=\textup{(\roman*)}]
  \item\label{it:primal-not-swap} $\varphi$ is not the local transposition of any pair of vertices of $F$ (which already follows from $|\Delta_{uv}| \ge 4$ via \Cref{cor:basic-ineq});
  \item\label{it:primal-not-vertex-cut} $\varphi$ is not an extension of a nontrivial automorphism of a small induced subgraph $F[D']$ (on at most three vertices) that fixes the remaining vertices;
  \item\label{it:primal-not-face-localised} $\varphi$ is not face-localised: $\supp(\varphi) \not\subseteq V(F^\circ)$ for any face $F^\circ$ of $F$;
  \item\label{it:primal-not-edge-cut} $\varphi$ does not decompose as $\alpha \cup \id_R$, where $\alpha$ is a nontrivial automorphism of an induced subgraph $F[C]$ with $C \subset \dom(\varphi)$, the remainder $R = \dom(\varphi) \setminus C$ is fixed pointwise with $|R| \ge 6$, and at most $5 - d(F)$ edges of $F$ join $C$ to $R$.
\end{enumerate}
\end{corollary}

\begin{proof}
Such a $\varphi$ has deficiency $k = d(F) \in \{2,3\}$, so \Cref{thm:unified-localisation} applies. Part~\ref{it:primal-not-face-localised} is clause~\ref{it:unified-face}, and part~\ref{it:primal-not-edge-cut} is \Cref{cor:unified-decomposable}.

For parts~\ref{it:primal-not-swap} and~\ref{it:primal-not-vertex-cut}, suppose $\varphi$ extends a nontrivial automorphism $\alpha$ of $F[D']$ with $|D'| \le 3$ by the identity on $D \setminus D'$; a local transposition of a pair $u,v$ is the case $D' = \{u,v\}$. Let $M \subseteq D'$ be the set of vertices moved by $\alpha$ and let $W$ be the set of fixed vertices of $F$ adjacent to $M$. By \Cref{lem:anchor}, $M$ carries either a transposition $u \leftrightarrow v$, in which case $W \subseteq N(u) \cap N(v)$ and $|W| \le 1$ by girth~$5$, or a $3$-cycle, in which case the unique possible element of $W$ is adjacent to all three moved vertices. Put $C := M \cup W$ and $R := \dom(\varphi) \setminus C$. No vertex of $M$ sends an edge to $R$, because its fixed neighbours all lie in $W$ and its moved neighbours lie in $M$; and the at most one vertex of $W$ sends at least two of its three edges into $M$, hence at most one edge to $R$. Thus at most one edge joins $C$ to $R$, while $|R| \ge n - k - 4 \ge 6$, and clause~\ref{it:unified-interface} applies.
\end{proof}

\begin{remark}[Hidden, almost-global symmetry]\label{rem:hidden-interpretation}
\Cref{cor:hidden-symmetry-primal} is a list of impossibilities, but its content is positive. The automorphism group of such a cage reports nothing at all, and yet all but two or three of its vertices can be matched with one another; and by the corollary no small part of the cage accounts for that matching. The near-symmetry is a property of the cage as a whole. Which symmetry it is, in every instance we have computed, is the subject of the conjecture in \Cref{sec:concluding}.
\end{remark}

\Cref{fig:asym_depth_2} illustrates \Cref{cor:hidden-symmetry-primal} for the smallest IPR fullerene of asymmetric depth~$2$.  The largest nontrivial partial automorphism is a near-mirror partial involution: it fixes the blue vertices (defining the axis of reflection), omits the two red vertices from its domain and range, and its support spans many faces of the cage.

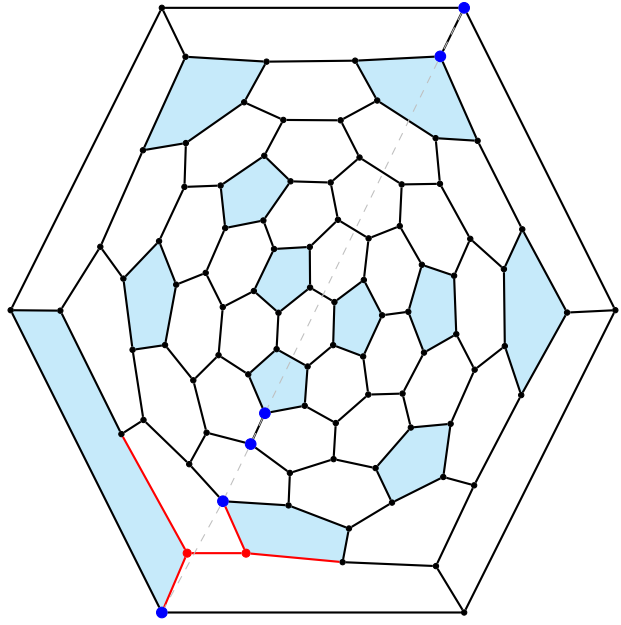
\begin{figure}[!h]
\centering
\begin{tikzpicture}[
    main_node/.style={circle, fill=black, draw=black, minimum size=2pt, inner sep=0pt},
    axis_node/.style={circle, fill=blue, draw=blue, minimum size=4pt, inner sep=0pt},
    highlight_node/.style={circle, fill=red, draw=red, minimum size=3pt, inner sep=0pt}
  ]

  \fill[cyan!20] (-1.467, 0.399) -- (-1.057, 0.169) -- (-1.095, -0.350) -- (-1.622, -0.449) -- (-1.842, 0.065) -- cycle;
  \fill[cyan!20] (-0.322, 0.302) -- (-0.719, 0.452) -- (-0.702, 1.022) -- (-0.313, 1.314) -- (-0.071, 0.847) -- cycle;
  \fill[cyan!20] (-1.028, 1.751) -- (-1.501, 1.724) -- (-1.768, 1.168) -- (-1.442, 0.880) -- (-1.024, 1.214) -- cycle;
  \fill[cyan!20] (-1.642, 2.101) -- (-2.147, 2.001) -- (-2.208, 2.563) -- (-1.630, 2.957) -- (-1.282, 2.619) -- cycle;
  \fill[cyan!20] (-2.795, 1.252) -- (-3.022, 1.827) -- (-3.494, 1.332) -- (-3.372, 0.390) -- (-2.943, 0.453) -- cycle;
  \fill[cyan!20] (-2.667, 3.125) -- (-1.896, 3.665) -- (-1.600, 4.202) -- (-2.672, 4.267) -- (-3.235, 3.030) -- cycle;
  \fill[cyan!20] (0.455, 1.515) -- (0.881, 1.370) -- (0.910, 0.597) -- (0.483, 0.353) -- (0.275, 0.893) -- cycle;
  \fill[cyan!20] (0.309, -0.638) -- (0.835, -0.589) -- (0.738, -1.292) -- (0.060, -1.631) -- (-0.158, -1.175) -- cycle;
  \fill[cyan!20] (-0.428, 4.215) -- (0.699, 4.272) -- (1.192, 3.155) -- (0.637, 3.190) -- (-0.136, 3.688) -- cycle;
  \fill[cyan!20] (-0.594, -2.419) -- (-1.871, -2.3) -- (-2.177, -1.613) -- (-1.309, -1.668) -- (-0.510, -1.974) -- cycle;
  \fill[cyan!20] (-2.985, -3.085) -- (-4.985, 0.914) -- (-4.328, 0.908) -- (-3.520, -0.725) -- (-2.65, -2.3) -- cycle;
  \fill[cyan!20] (1.783, 1.984) -- (2.376, 0.883) -- (1.769, -0.209) -- (1.552, 0.439) -- (1.540, 1.458) -- cycle;

  \node[main_node] (0) at (3.014, 0.914) {};
  \node[axis_node] (1) at (1.014, 4.914) {};
  \node[main_node] (2) at (2.376, 0.883) {};
  \node[main_node] (3) at (1.014, -3.085) {};
  \node[main_node] (4) at (-2.985, 4.914) {};
  \node[axis_node] (5) at (0.699, 4.272) {};
  \node[main_node] (6) at (1.783, 1.984) {};
  \node[main_node] (7) at (1.769, -0.209) {};
  \node[main_node] (8) at (0.641, -2.471) {};
  \node[axis_node] (9) at (-2.985, -3.085) {};
  \node[main_node] (10) at (-4.985, 0.914) {};
  \node[main_node] (11) at (-2.672, 4.267) {};
  \node[main_node] (12) at (-0.428, 4.215) {};
  \node[main_node] (13) at (1.192, 3.155) {};
  \node[main_node] (14) at (1.540, 1.458) {};
  \node[main_node] (15) at (1.552, 0.439) {};
  \node[main_node] (16) at (1.146, -1.403) {};
  \node[main_node] (17) at (-0.594, -2.419) {};
  \node[highlight_node] (18) at (-2.65, -2.3) {};
  \node[main_node] (19) at (-4.328, 0.908) {};
  \node[main_node] (20) at (-3.235, 3.030) {};
  \node[main_node] (21) at (-1.600, 4.202) {};
  \node[main_node] (22) at (-0.136, 3.688) {};
  \node[main_node] (23) at (0.637, 3.190) {};
  \node[main_node] (24) at (1.095, 1.858) {};
  \node[main_node] (25) at (1.153, 0.127) {};
  \node[main_node] (26) at (0.738, -1.292) {};
  \node[main_node] (27) at (-0.510, -1.974) {};
  \node[highlight_node] (28) at (-1.871, -2.3) {};
  \node[main_node] (29) at (-3.520, -0.725) {};
  \node[main_node] (30) at (-3.799, 1.753) {};
  \node[main_node] (31) at (-2.667, 3.125) {};
  \node[main_node] (32) at (-1.896, 3.665) {};
  \node[main_node] (33) at (-0.620, 3.426) {};
  \node[main_node] (34) at (0.694, 2.586) {};
  \node[main_node] (35) at (0.881, 1.370) {};
  \node[main_node] (36) at (0.910, 0.597) {};
  \node[main_node] (37) at (0.835, -0.589) {};
  \node[main_node] (38) at (0.060, -1.631) {};
  \node[main_node] (39) at (-1.309, -1.668) {};
  \node[axis_node] (40) at (-2.177, -1.613) {};
  \node[main_node] (41) at (-3.228, -0.538) {};
  \node[main_node] (42) at (-3.494, 1.332) {};
  \node[main_node] (43) at (-2.686, 2.545) {};
  \node[main_node] (44) at (-1.378, 3.431) {};
  \node[main_node] (45) at (-0.370, 2.934) {};
  \node[main_node] (46) at (0.188, 2.578) {};
  \node[main_node] (47) at (0.455, 1.515) {};
  \node[main_node] (48) at (0.483, 0.353) {};
  \node[main_node] (49) at (0.309, -0.638) {};
  \node[main_node] (50) at (-0.158, -1.175) {};
  \node[main_node] (51) at (-1.291, -1.239) {};
  \node[main_node] (52) at (-2.623, -1.123) {};
  \node[main_node] (53) at (-3.372, 0.390) {};
  \node[main_node] (54) at (-3.022, 1.827) {};
  \node[main_node] (55) at (-2.208, 2.563) {};
  \node[main_node] (56) at (-1.630, 2.957) {};
  \node[main_node] (57) at (-0.752, 2.603) {};
  \node[main_node] (58) at (0.162, 2.027) {};
  \node[main_node] (59) at (0.275, 0.893) {};
  \node[main_node] (60) at (0.199, -0.189) {};
  \node[main_node] (61) at (-0.713, -1.056) {};
  \node[axis_node] (62) at (-1.810, -0.857) {};
  \node[main_node] (63) at (-2.393, -0.706) {};
  \node[main_node] (64) at (-2.943, 0.453) {};
  \node[main_node] (65) at (-2.795, 1.252) {};
  \node[main_node] (66) at (-2.147, 2.001) {};
  \node[main_node] (67) at (-1.282, 2.619) {};
  \node[main_node] (68) at (-0.655, 2.110) {};
  \node[main_node] (69) at (-0.250, 1.864) {};
  \node[main_node] (70) at (-0.071, 0.847) {};
  \node[main_node] (71) at (-0.254, -0.202) {};
  \node[main_node] (72) at (-0.684, -0.579) {};
  \node[axis_node] (73) at (-1.622, -0.449) {};
  \node[main_node] (74) at (-2.569, -0.013) {};
  \node[main_node] (75) at (-2.405, 1.408) {};
  \node[main_node] (76) at (-1.642, 2.101) {};
  \node[main_node] (77) at (-1.028, 1.751) {};
  \node[main_node] (78) at (-0.313, 1.314) {};
  \node[main_node] (79) at (-0.322, 0.302) {};
  \node[main_node] (80) at (-1.095, -0.350) {};
  \node[main_node] (81) at (-1.842, 0.065) {};
  \node[main_node] (82) at (-2.235, 0.319) {};
  \node[main_node] (83) at (-2.179, 0.955) {};
  \node[main_node] (84) at (-1.501, 1.724) {};
  \node[main_node] (85) at (-1.024, 1.214) {};
  \node[main_node] (86) at (-0.702, 1.022) {};
  \node[main_node] (87) at (-0.719, 0.452) {};
  \node[main_node] (88) at (-1.057, 0.169) {};
  \node[main_node] (89) at (-1.467, 0.399) {};
  \node[main_node] (90) at (-1.768, 1.168) {};
  \node[main_node] (91) at (-1.442, 0.880) {};

  \path[draw, thick]
  (0) edge (1) (0) edge (2) (0) edge (3)
  (1) edge (4) (1) edge (5)
  (2) edge (6) (2) edge (7)
  (3) edge (8) (3) edge (9)
  (4) edge (10) (4) edge (11)
  (5) edge (12) (5) edge (13)
  (6) edge (13) (6) edge (14)
  (7) edge (15) (7) edge (16)
  (8) edge (16) (8) edge (17)
  (9) edge (10)
  (9) edge[red] (18) 
  (10) edge (19)
  (11) edge (20) (11) edge (21)
  (12) edge (21) (12) edge (22)
  (13) edge (23)
  (14) edge (15) (14) edge (24)
  (15) edge (25)
  (16) edge (26)
  (17) edge (27)
  (17) edge[red] (28) 
  (18) edge[red] (28) 
  (18) edge[red] (29) 
  (19) edge (29) 
  (19) edge (30)
  (20) edge (30) (20) edge (31)
  (21) edge (32)
  (22) edge (23) (22) edge (33)
  (23) edge (34)
  (24) edge (34) (24) edge (35)
  (25) edge (36) (25) edge (37)
  (26) edge (37) (26) edge (38)
  (27) edge (38) (27) edge (39)
  (28) edge[red] (40) 
  (29) edge (41)
  (30) edge (42)
  (31) edge (32) (31) edge (43)
  (32) edge (44)
  (33) edge (44) (33) edge (45)
  (34) edge (46)
  (35) edge (36) (35) edge (47)
  (36) edge (48)
  (37) edge (49)
  (38) edge (50)
  (39) edge (40) (39) edge (51)
  (40) edge (52)
  (41) edge (52) (41) edge (53)
  (42) edge (53) (42) edge (54)
  (43) edge (54) (43) edge (55)
  (44) edge (56)
  (45) edge (46) (45) edge (57)
  (46) edge (58)
  (47) edge (58) (47) edge (59)
  (48) edge (59) (48) edge (60)
  (49) edge (50) (49) edge (60)
  (50) edge (61)
  (51) edge (61) (51) edge (62)
  (52) edge (63)
  (53) edge (64)
  (54) edge (65)
  (55) edge (56) (55) edge (66)
  (56) edge (67)
  (57) edge (67) (57) edge (68)
  (58) edge (69)
  (59) edge (70)
  (60) edge (71)
  (61) edge (72)
  (62) edge (63) (62) edge (73)
  (63) edge (74)
  (64) edge (65) (64) edge (74)
  (65) edge (75)
  (66) edge (75) (66) edge (76)
  (67) edge (76)
  (68) edge (69) (68) edge (77)
  (69) edge (78)
  (70) edge (78) (70) edge (79)
  (71) edge (72) (71) edge (79)
  (72) edge (80)
  (73) edge (80) (73) edge (81)
  (74) edge (82)
  (75) edge (83)
  (76) edge (84)
  (77) edge (84) (77) edge (85)
  (78) edge (86)
  (79) edge (87)
  (80) edge (88)
  (81) edge (82) (81) edge (89)
  (82) edge (83)
  (83) edge (90)
  (84) edge (90)
  (85) edge (86) (85) edge (91)
  (86) edge (87)
  (87) edge (88)
  (88) edge (89)
  (89) edge (91)
  (90) edge (91);

  \draw[dashed, gray!50, thin] (1) -- (5) -- (73) -- (62) -- (40) -- (9);

\end{tikzpicture}
\caption{The smallest IPR fullerene of asymmetric depth $2$, on $92$ vertices, illustrating \Cref{cor:hidden-symmetry-primal}. The partial automorphism that realises $d(F) = 2$ captures the mirror symmetry that fixes the blue vertices and omits the two red vertices from its domain and range.}
\label{fig:asym_depth_2}
\end{figure}

As can be seen in Table~\ref{table:asym_ipr}, all of these remaining values of asymmetric depth are obtained by some fullerenes. 

In view of \Cref{cor:hidden-symmetry-primal} and the data in \Cref{table:asym_ipr}, most IPR fullerenes attain the maximum asymmetric depth $4$ via a partial automorphism that is essentially a trivial transposition.
By $n=118$, the depth-$4$ column accounts for $\nicefrac{7546}{7670} \approx 98.4\%$ of asymmetric IPR fullerenes. In contrast, IPR fullerenes of asymmetric depth $2$ or $3$ exhibit hidden, almost-global symmetries, as illustrated by the example in \Cref{fig:asym_depth_2}. Inspecting our computational data, we observe that similar hidden symmetries occur for every depth-$2$ or depth-$3$ IPR fullerene up to $n = 118$. The smallest such fullerene of depth $2$ (resp.\ $3$) has $92$ (resp.\ $84$) vertices.

\begin{remark}
An analogous analysis can be carried out for the duals of IPR fullerenes. Call a separating cycle in a planar triangulation \emph{trivial} if its vertex set is the open neighbourhood $N(p)$ of a single vertex $p$. Under planar duality, cyclic $k$-edge cuts of $F$ correspond to separating $k$-cycles of $\dual(F)$, and the trivial cyclic cuts --- those isolating a single face of $F$ --- correspond exactly to the trivial separating cycles. Since fullerene graphs are cyclically $5$-edge-connected, $\dual(F)$ contains no separating $3$- or $4$-cycles, and by the Kardo\v{s}--\v{S}krekovski result recalled at the start of \S\ref{subsec:localised}, in the dual of an IPR fullerene every separating $5$- or $6$-cycle is trivial. This high cyclic connectivity of the dual allows one to obtain a dual analogue of \Cref{cor:hidden-symmetry-primal}, ruling out localised partial automorphisms of small asymmetric depth in $\dual(F)$.
\end{remark}

\begin{table}
\centering
\small
\setlength{\tabcolsep}{6pt}
\renewcommand{\arraystretch}{1.1}
\begin{tabular}{@{}r r r r r r r r@{}}
  \toprule
  nv & Total IPR & Asymmetric IPR & \% Asymmetric & non-asymmetric & depth 2 & depth 3 & depth 4 \\
  \midrule
  60  & 1       & 0       & 0.0  & 1   & 0  & 0  & 0   \\
  62  & 0       & 0       & 0.0  & 0   & 0  & 0  & 0   \\
  64  & 0       & 0       & 0.0  & 0   & 0  & 0  & 0   \\
  66  & 0       & 0       & 0.0  & 0   & 0  & 0  & 0   \\
  68  & 0       & 0       & 0.0  & 0   & 0  & 0  & 0   \\
  70  & 1       & 0       & 0.0  & 1   & 0  & 0  & 0   \\
  72  & 1       & 0       & 0.0  & 1   & 0  & 0  & 0   \\
  74  & 1       & 0       & 0.0  & 1   & 0  & 0  & 0   \\
  76  & 2       & 0       & 0.0  & 2   & 0  & 0  & 0   \\
  78  & 5       & 0       & 0.0  & 5   & 0  & 0  & 0   \\
  80  & 7       & 0       & 0.0  & 7   & 0  & 0  & 0   \\
  82  & 9       & 0       & 0.0  & 9   & 0  & 0  & 0   \\
  84  & 24      & 1       & 4.2  & 23  & 0  & 1  & 0   \\
  86  & 19      & 6       & 31.6 & 13  & 0  & 2  & 4   \\
  88  & 35      & 11      & 31.4 & 24  & 0  & 2  & 9   \\
  90  & 46      & 16      & 34.8 & 30  & 0  & 4  & 12  \\
  92  & 86      & 38      & 44.2 & 48  & 3  & 1  & 34  \\
  94  & 134     & 89      & 66.4 & 45  & 4  & 6  & 79  \\
  96  & 187     & 108     & 57.8 & 79  & 2  & 11 & 95  \\
  98  & 259     & 169     & 65.3 & 90  & 3  & 12 & 154 \\
  100 & 450     & 336     & 74.7 & 114 & 2  & 24 & 310 \\
  102 & 616     & 488     & 79.2 & 128 & 10 & 20 & 458 \\
  104 & 823     & 644     & 78.3 & 179 & 6  & 26 & 612 \\
  106 & 1233    & 1054    & 85.5 & 179 & 10 & 38 & 1006 \\
  108 & 1799    & 1479    & 82.2 & 320 & 8  & 52 & 1419 \\
  110 & 2355    & 2111    & 89.6 & 244 & 16 & 54 & 2041 \\
  112 & 3342    & 2950    & 88.3 & 392 & 27 & 44 & 2879 \\
  114 & 4468    & 4089    & 91.5 & 379 & 15 & 88 & 3986 \\
  116 & 6063    & 5508    & 90.8 & 555 & 27 & 85 & 5396 \\
  118 & 8148    & 7670    & 94.1 & 478 & 14 & 110 & 7546 \\
  \bottomrule
\end{tabular}
\caption{Number of IPR and asymmetric IPR fullerenes and their ratio. The \emph{non-asymmetric} column is exactly the asymmetric depth-$0$ (symmetric) count; asymmetric depth~$1$ never occurs (\Cref{thm:fullerene-depth-gt-1}) and is therefore omitted; the columns \emph{depth 2}, \emph{depth 3}, \emph{depth 4} give the number of asymmetric IPR fullerenes attaining each value, so that non-asymmetric $+$ depth $2 +$ depth $3 +$ depth $4$ equals the total IPR count.}
\label{table:asym_ipr}
\end{table}

\section{Graphs of higher genus: upper bound}

The bound of this section is an easy consequence of the degeneracy of graphs embedded on a surface, and is included for completeness; unlike the planar bound of \Cref{t:planar} it is almost certainly far from tight, and already at $g = 1$ it yields only $d(F) \le 12$.

Below we generalise the planar bound of \Cref{t:planar} via the notion of degeneracy. In contrast to the tight planar bound, we do not know it to be tight (\Cref{q:genus-tight}).

Recall that an orientable surface of \emph{genus} $g$ is the sphere with $g$ handles attached, and that the \emph{genus} $g(F)$ of a graph $F$ is the smallest $g$ for which $F$ embeds on the orientable surface of genus $g$ so that edges meet only at their endpoints. Thus $g(F) = 0$ exactly for planar graphs, and larger $g(F)$ measures how far $F$ is from being planar. Throughout this section $g$ denotes the genus of $F$.

A graph $F$ has \emph{degeneracy} at most $k$, $\degen(F) \le k$, if \emph{every} induced subgraph
contains a vertex of degree $\le k$. The following is a classical algorithm~\cite{matula1983smallest} used to determine the degeneracy of a graph.

\begin{algorithm}[h]
\caption{Peeling algorithm for degeneracy ordering.}\label{alg:peeling}
\begin{algorithmic}[1]
\While{the current subgraph has vertices}
  \State find a vertex of minimum degree (guaranteed $\le k$) in the current subgraph;
  \State remove it and place it next in the ordering.
\EndWhile
\end{algorithmic}
\end{algorithm}

\noindent The result is an ordering $v_1,\dots,v_n$ such that each $v_i$ has at most $k$ neighbours among $v_{i+1},\dots,v_n$.

If a graph $F$ has degeneracy $k$, then in the ordering produced by Algorithm~\ref{alg:peeling}, $v_1$ has all its neighbours later in the ordering, so $|N(v_1)| \le k$. The vertex $v_2$ has at most $k$ neighbours later in the ordering, plus possibly $v_1$; in particular $|N(v_2) \setminus \{v_1\}| \le k$. Recalling that $\Delta_{v_1 v_2} = (N(v_1) \setminus \{v_2\}) \mathbin{\triangle} (N(v_2) \setminus \{v_1\})$, we obtain
\[
|\Delta_{v_1 v_2}|
  \le |N(v_1) \setminus \{v_2\}| + |N(v_2) \setminus \{v_1\}|
  \le k + k = 2k.
\]

Therefore, if a graph has degeneracy at most $k$, then it has symmetric difference at most $2k$.

\begin{equation}\label{ineq:degen}
\min_{1 \leq i < j \leq n} \{ |\Delta_{ij}| \}\le 2\,\degen(F)
\end{equation}

In what follows, we improve the bound for graphs of a given genus. Recall the Heawood number, which bounds the number of colors sufficient for map coloring on a surface of genus $g$ for $g>0$.

\begin{equation}\label{heawood}
H(g)=\left\lfloor \dfrac{7+\sqrt{1+48g}}{2}\right\rfloor    
\end{equation}

Moreover, recall that degeneracy is bounded by Heawood number.

\begin{equation}\label{heawood:degeneracy}
     \degen(F)\le H(g)-1
\end{equation}

 \begin{theorem}\label{thm:genus-bound}
Let $F$ be a graph of genus $g$, $g>0$. Then
\[
  d(F) \;\le\; 2H(g)-2 \;=\; 2\!\left\lfloor \frac{7+\sqrt{1+48g}}{2}\right\rfloor - 2.
\]
In particular, $d(F) \le \left\lfloor 5+\sqrt{1+48g}\right\rfloor$ as a weaker but explicit closed-form bound.
 \end{theorem}

\begin{proof}
By~\eqref{basic-ineq} and~\eqref{ineq:degen},
\(d(F) \le 2\,\degen(F)\).
By~\eqref{heawood:degeneracy}, $\degen(F)\le H(g)-1$, so
\[
  d(F) \le 2(H(g)-1) = 2H(g)-2.
\]
Since $H(g)=\lfloor(7+\sqrt{1+48g})/2\rfloor \le (7+\sqrt{1+48g})/2$ and $d(F)$ is an integer, we obtain
$2H(g)-2 \le \lfloor 5+\sqrt{1+48g}\rfloor$ as a weaker but explicit closed-form bound.
\end{proof}

\begin{remark}
While \Cref{thm:genus-bound} establishes an upper bound dependent on the genus $g$, this bound is tightest for small graphs that heavily wrap around the surface. For a fixed genus $g$, as the number of vertices $n$ grows arbitrarily large, the structure behaves locally like a planar graph. By the generalised Euler formula, a graph of genus $g$ has at most $3n - 6 + 6g$ edges. Thus, its average degree is bounded by $6 + \frac{12g - 12}{n}$. As $n \to \infty$, the average degree approaches $6$, so discharging arguments analogous to those used in the planar case \cite{aksionovDeeplyAsymmetricPlanar2005} again guarantee vertices of small degree, allowing local transpositions that restrict asymmetric depth. In other words, to attain higher asymmetric depth one is typically forced to higher genus.
\end{remark}

\section{Concluding remarks}\label{sec:concluding}

The data of \Cref{table:asym_ipr}, together with \Cref{fig:asym_depth_2}, suggest that the non-localisation of \Cref{cor:hidden-symmetry-primal} has a concrete geometric cause: in every low-depth cage we have computed, the partial automorphism realising the depth is a reflection of the cage, broken at two or three atoms. To state this precisely we make the word \emph{mirror} exact.

Recall that a $3$-connected planar graph has an essentially unique embedding in the sphere, so $F$ carries a canonical rotation system, determined up to a simultaneous reversal at every vertex. Let $\varphi$ be a partial automorphism of $F$ with $D := \dom(\varphi)$. Call a vertex $x \in D$ \emph{interior} for $\varphi$ if $N(x) \subseteq D$. For such an $x$ the map $\varphi$ carries $N(x)$ bijectively onto $N(\varphi(x))$, and since both are triples the induced map either preserves or reverses the cyclic order supplied by the rotation system; reversing all rotations simultaneously does not change which of the two occurs, so the distinction does not depend on the choice of embedding. A partial automorphism of deficiency $k$ has at least $n - 4k$ interior vertices, so for the cages considered here the following notion is never vacuous.

\begin{definition}[Near-mirror partial automorphism]\label{def:near-mirror}
A partial automorphism $\varphi$ of a fullerene $F$ is a \emph{near-mirror} if it is self-inverse, in the sense that $\varphi(\varphi(x)) = x$ whenever $x$ and $\varphi(x)$ both lie in $\dom(\varphi)$, and if it reverses the cyclic order of neighbours at every interior vertex.
\end{definition}

\begin{conjecture}\label{conj:mirror}
Let $F$ be an \emph{asymmetric} IPR fullerene with $d(F) \in \{2,3\}$. Then $d(F)$ is realised by a near-mirror partial automorphism.
\end{conjecture}

We have verified \Cref{conj:mirror} for all $727$ asymmetric IPR fullerenes of asymmetric depth $2$ or $3$ on at most $118$ vertices ($147$ of depth $2$ and $580$ of depth $3$; see \Cref{table:asym_ipr}). The hypothesis that $F$ be asymmetric cannot be dropped: a symmetric IPR fullerene has $d(F) = 0$ and still admits near-mirror partial automorphisms of positive deficiency.

We tried starting from some of these IPR fullerenes of asymmetric depth $2$ and $3$ and applying simple patch replacement operations described in~\cite{andovaMathematicalAspectsFullerenes2016} to recover the underlying global mirror symmetries, but we did not succeed. Do there exist fullerene patch replacement operations that allow one to recover the underlying global symmetry?

Moreover, the relationship between the asymmetric depth of a fullerene and that of its dual is not yet clear. There are examples where the two depths are equal and examples where the dual depth is greater or lesser than the depth of the fullerene.  Across all fullerenes of order at most~$118$ enumerated by \texttt{buckygen}, our computations show $|d(F) - d(F^*)| \le 2$, motivating the following. We conjecture that for every fullerene graph $F$, $|d(F) - d(F^*)| \le 2$.

It is unclear whether the bound for graphs of higher genus is tight. We have a construction of asymmetric graphs using columns of hypercubes described in \cite{cingelPartialAutomorphismsLevel2024} that yields a graph of asymmetric depth $d$ for any $d \ge 0$, with increasing genus but the genus grows too fast to match the bound. We record this as an open problem.

\begin{question}\label{q:genus-tight}
Is the bound of \Cref{thm:genus-bound} tight? That is, are there infinitely many $g$ for which some graph of genus~$g$ attains $d(F) = 2H(g) - 2$, or even $d(F) = \Theta(\sqrt{g})$?
\end{question}

\section{Computational methodology}\label{sec:computational-methodology}

\paragraph{Generation of fullerenes.}
All fullerene graphs used in our computational experiments were generated using \texttt{buckygen} \cite{buckygen}. In particular, for each even $n$ in the tested range, we generated the complete list of nonisomorphic IPR fullerenes on $n$ vertices.

\paragraph{Computing asymmetry and asymmetric depth.}
Automorphism and isomorphism computations were carried out with the \texttt{nauty}/\texttt{Traces} suite, called from \textsc{Julia} version~$1.12$. For each fullerene $F$ (and its dual $F^*$), we invoked \texttt{nauty}~2.9.1 to test whether the automorphism group is trivial. To determine the asymmetric depth, we searched for large-rank nontrivial partial automorphisms by identifying isomorphisms between large induced subgraphs---in particular, between vertex-deleted subgraphs---via canonical hashes. The largest rank $k$ at which a nontrivial partial automorphism is found determines $d(F)=|V(F)|-k$.

\textbf{Computational resources.} 
The main workload to compute asymmetric depth of fullerenes up to 118 vertices ran on high-performance computing (HPC) cluster with compute nodes equipped with 2 × AMD EPYC 9745 (each with 256 cores / 512 threads per node, operating at 2.4 GHz (3.7 GHz Turbo)) CPUs and 1,536 GB DDR5 RAM. In total, computations took roughly 4 CPU months to complete.

\textbf{Correctness.} Since some results in this paper rely on the outcomes of algorithms, it is important to take extra measures to ensure that the algorithms were implemented correctly. We follow the guidelines in \cite{jookenComputerassistedGraphTheory2025}. Because our approach checks, by brute force, all partial automorphisms of the target ranks, correctness, and completeness follow directly. Nevertheless, since our implementation relies on external isomorphism software, we also implemented independent versions using different packages for computing isomorphisms (\texttt{nauty} and \texttt{NetworkX} VF2) in two different languages (\textsc{Julia} and \textsc{Python}). The results were always in agreement. Moreover, we used unit tests on small graphs that we also verified by hand.

\paragraph{The exhaustive search of \Cref{lem:acyclic-case}.}
The finite verification underlying \Cref{lem:acyclic-case} is implemented in \textsc{Python}~3 using \texttt{NetworkX} for graph isomorphism (VF2) and non-isomorphic tree generation. The program enumerates certificates exactly as described in \S\ref{subsec:unified} and applies the conditions~\ref{A1}--\ref{A5} as independent filters; the totals quoted in the proof are produced by the run
\texttt{search\_exit2.py $k$ $c_{\max}$ $m_{\min}$ $m_{\max}$} for $(k, c_{\max}) \in \{(1,4), (2,3), (3,2)\}$ and $m$ ranging over $1 \le m \le 3 + 2k$. The whole verification takes under two minutes on a single core.

\textbf{Correctness.} As in the enumeration of asymmetric depth, we follow the guidelines of~\cite{jookenComputerassistedGraphTheory2025}. The search is a complete enumeration of a finite space, so correctness reduces to the faithfulness of the filters and the exhaustiveness of the enumeration. Three independent checks were carried out. First, the case $\dom(\varphi) = \ran(\varphi)$ was implemented twice, once via automorphisms of the induced subgraph $F[C]$ and once directly from the definition of a partial automorphism over all labelled sets $C$ with $|C| \le 6$; the two implementations agree on the number of records reaching the face conditions and on their isomorphism classes. Second, the completeness of the forest enumeration rests on an independent trusted generator: the subcubic trees from which the forests $F[C]$ are assembled were generated with \texttt{gentreeg} (part of the \texttt{nauty} suite), which agrees with our own enumeration in every isomorphism class up to nine vertices, and re-running the entire search on the \texttt{gentreeg}-generated trees reproduces the outcome exactly; as a further check the tree counts match the orbit identity $\sum_{T} |C|!\,/\,|\Aut(T)|$. Third, the cycle enumerator and the face conditions were unit-tested on graphs verified by hand; for instance, the Petersen graph is correctly reported to have twelve $5$-cycles and ten $6$-cycles, and a pair of pentagons sharing an edge is correctly rejected. An earlier version of the search that pruned the labelling of $S_D$ incorrectly was detected by the first of these checks.

\paragraph{The graph of \Cref{prop:girth5-witness}.}
The witness graph is given explicitly, together with a drawing, in \Cref{app:witness}. We verified with \texttt{NetworkX} that it is cubic, $3$-connected, non-planar, of girth~$5$ and with trivial automorphism group, and that the stated map is a partial automorphism of rank~$17$.

\paragraph{Data.}
All explicitly mentioned or shown graphs are also available in the House of Graphs database \cite{coolsaet2023house} by searching for the term ``asymmetric depth''. The source code for the asymmetric-depth computation, the finite enumeration underlying \Cref{lem:acyclic-case} (\texttt{search\_exit2.py}), and the verification of the graph of \Cref{prop:girth5-witness} are deposited in a public repository.
\footnote{Archived at \url{https://github.com/JanPastorek/asym_depth_fullerenes}.}

\appendix

\section{The witness graph of \Cref{prop:girth5-witness}}\label{app:witness}

On the vertex set $\{x,u,w,v,y\} \cup \{\rho\} \cup \{s_1,s_2,s_3\} \cup \{e_0, \dots, e_{10}\}$, the thirty edges
\[
\begin{array}{l}
xu,\; uw,\; wv,\; vy,\; \rho x,\; \rho y,\; s_1u,\; s_1y,\; s_2v,\; s_2x,\; s_3w, \\[2pt]
e_0e_3,\; e_0e_4,\; e_0s_3,\; e_1e_4,\; e_1e_7,\; e_1s_2,\; e_2e_3,\; e_2e_7,\; e_2e_9,\; e_3\rho, \\[2pt]
e_4e_5,\; e_5e_6,\; e_5e_9,\; e_6e_{10},\; e_6s_1,\; e_7e_{10},\; e_8e_9,\; e_8e_{10},\; e_8s_3
\end{array}
\]
define a cubic graph $G$ on $20$ vertices, drawn in \Cref{fig:witness-graph}. It is $3$-connected, non-planar, has girth~$5$ and trivial automorphism group. The map
\[
  \varphi \;=\; (x\,y)(u\,v) \;\cup\; \id_{V(G) \setminus \{x,u,v,y,s_1,s_2,s_3\}}
\]
is a partial automorphism of $G$ of rank~$17$, hence of deficiency $k = 3$. Its support is $\{x,u,v,y\}$, it is contained in $C = \{x,u,w,v,y\}$, and the only edges joining $C$ to the pointwise-fixed remainder $R = V(G) \setminus (C \cup \{s_1,s_2,s_3\})$, which has twelve vertices, are $\rho x$ and $\rho y$. Thus $\varphi$ is localised behind a two-edge interface, which is exactly what \Cref{thm:unified-localisation}\ref{it:unified-interface} forbids in an IPR fullerene; $G$ shows that planarity, and not merely cubicity and girth~$5$, is doing the work there.

\begin{figure}[!ht]
\centering
\begin{tikzpicture}[scale=0.78,
  every node/.style={font=\small},
  mv/.style ={circle, draw=black, fill=black, text=white, inner sep=1.2pt, minimum size=6.2mm},
  fx/.style ={circle, draw=black, fill=white, inner sep=1.2pt, minimum size=6.2mm},
  del/.style={rectangle, draw=black, fill=gray!35, inner sep=2.4pt, minimum size=6.2mm},
  ed/.style    ={draw=black!55, line width=0.6pt},
  cedge/.style ={draw=black, line width=1.5pt},
  iface/.style ={draw=black, line width=1.2pt, dashed}]

  \draw[ed] (-3.53,2.83)--(-5.72,0.89) (-3.53,2.83)--(-4.53,-0.79)
            (-1.20,-2.58)--(-1.93,-0.80) (-1.20,-2.58)--(-5.74,-3.03)
            (-0.05,2.42)--(-2.01,1.34)
            (0.12,-0.74)--(0.75,-3.67) (0.12,-0.74)--(1.96,-0.05) (0.12,-0.74)--(-0.05,2.42)
            (2.03,-2.00)--(1.96,-0.05) (2.03,-2.00)--(3.78,-0.72) (2.03,-2.00)--(-1.20,-2.58)
            (5.30,-2.45)--(0.75,-3.67) (5.30,-2.45)--(3.78,-0.72) (5.30,-2.45)--(5.02,1.15)
            (0.75,-3.67)--(-3.27,-4.71)
            (1.96,-0.05)--(2.41,1.88) (2.41,1.88)--(0.77,4.76) (2.41,1.88)--(5.02,1.15)
            (0.77,4.76)--(3.75,5.55) (0.77,4.76)--(-3.53,2.83)
            (3.78,-0.72)--(3.75,5.55)
            (5.13,3.90)--(5.02,1.15) (5.13,3.90)--(3.75,5.55) (5.13,3.90)--(-0.05,2.42);
  \draw[cedge] (-5.74,-3.03)--(-5.72,0.89) (-5.72,0.89)--(-2.01,1.34)
               (-2.01,1.34)--(-1.93,-0.80) (-1.93,-0.80)--(-4.53,-0.79);
  \draw[iface] (-3.27,-4.71)--(-5.74,-3.03) (-3.27,-4.71)--(-4.53,-0.79);

  \node[mv]  (x)   at (-5.74,-3.03) {$x$};
  \node[mv]  (u)   at (-5.72,0.89)  {$u$};
  \node[fx]  (w)   at (-2.01,1.34)  {$w$};
  \node[mv]  (v)   at (-1.93,-0.80) {$v$};
  \node[mv]  (y)   at (-4.53,-0.79) {$y$};
  \node[fx]  (rh)  at (-3.27,-4.71) {$\rho$};
  \node[del] (s1)  at (-3.53,2.83)  {$s_1$};
  \node[del] (s2)  at (-1.20,-2.58) {$s_2$};
  \node[del] (s3)  at (-0.05,2.42)  {$s_3$};
  \node[fx]  (e0)  at (0.12,-0.74)  {$e_{0}$};
  \node[fx]  (e1)  at (2.03,-2.00)  {$e_{1}$};
  \node[fx]  (e2)  at (5.30,-2.45)  {$e_{2}$};
  \node[fx]  (e3)  at (0.75,-3.67)  {$e_{3}$};
  \node[fx]  (e4)  at (1.96,-0.05)  {$e_{4}$};
  \node[fx]  (e5)  at (2.41,1.88)   {$e_{5}$};
  \node[fx]  (e6)  at (0.77,4.76)   {$e_{6}$};
  \node[fx]  (e7)  at (3.78,-0.72)  {$e_{7}$};
  \node[fx]  (e8)  at (5.13,3.90)   {$e_{8}$};
  \node[fx]  (e9)  at (5.02,1.15)   {$e_{9}$};
  \node[fx]  (e10) at (3.75,5.55)   {$e_{10}$};

  \begin{scope}[shift={(-5.9,6.3)}, every node/.style={font=\scriptsize}]
    \node[mv,  minimum size=4.6mm, inner sep=0.6pt] at (0,0) {};
    \node[anchor=west] at (0.35,0) {moved by $\varphi$};
    \node[fx,  minimum size=4.6mm, inner sep=0.6pt] at (3.6,0) {};
    \node[anchor=west] at (3.95,0) {fixed by $\varphi$};
    \node[del, minimum size=4.6mm, inner sep=1.2pt] at (7.0,0) {};
    \node[anchor=west] at (7.4,0) {deleted ($S_D$)};
  \end{scope}
\end{tikzpicture}
\caption{The graph $G$ of \Cref{prop:girth5-witness}: cubic, $3$-connected, of girth~$5$, non-planar, and asymmetric. Thick edges form the induced path $C = x\,u\,w\,v\,y$ carrying the reversal $(x\,y)(u\,v)$; the two dashed edges $\rho x$ and $\rho y$ are the entire interface between $C$ and the twelve-vertex pointwise-fixed remainder $R$. The three squares are the deleted vertices $s_1, s_2, s_3$. Crossings in the drawing are artefacts of the layout, $G$ being non-planar.}
\label{fig:witness-graph}
\end{figure}
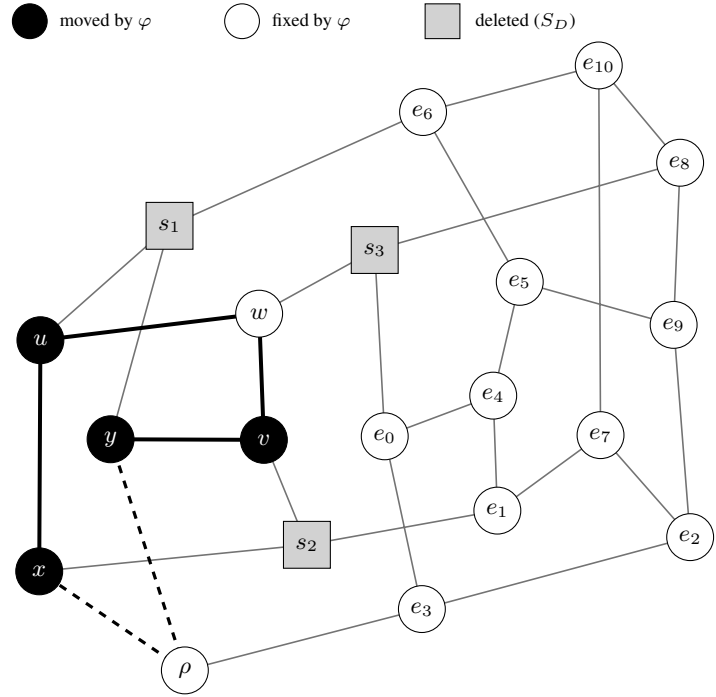

\section*{Acknowledgements}

The author gratefully acknowledges valuable discussions at Ko\v{s}ický kombinatorický seminár and subsequent discussions with Roman Soták, Tomá\v{s} Madaras and Franti\v{s}ek Kardo\v{s}.

This work was supported by the use of computational resources of the supercomputer PERUN, operated by the Supercomputing Centre at the Technical University of Košice (TUKE), Slovakia with the support of the European Union from the funds of the Recovery and Resilience Plan of the Slovak Republic within the framework of project No. 17I03-04-P03-00001, Development and design of a supercomputer for the National Supercomputing Center.

The author is supported by \textit{Agent\'{u}ra na podporu v\'{y}skumu a v\'{y}voja} (APVV grant SK-AT-23-0019), by \textit{Vedeck\'{a} grantov\'{a} agent\'{u}ra} (VEGA grant 1/0437/23) and by \textit{Comenius University} (grant UK/1020/2026).

\section*{Declaration on the use of AI tools}

During the preparation of this work the author used generative artificial-intelligence assistants (large language models) for two purposes: to improve the presentation of the manuscript---grammar, wording, and stylistic consistency---and as an additional reviewer of the content, checking statements, cross-references, and the internal consistency of the reported figures, and suggesting revisions. Every suggestion produced by these tools was reviewed and verified by the author, who takes full responsibility for the content of this publication.

\printbibliography

\end{document}